\documentclass[11pt,reqno]{amsart}

\usepackage{amssymb}
\usepackage{latexsym}
\usepackage{amsbsy}
\usepackage{amsfonts}
\usepackage{appendix}
\usepackage{pstricks}
\usepackage{hyperref}
\usepackage{color}
\allowdisplaybreaks
\usepackage{enumitem}

\numberwithin{equation}{section}
\newtheorem{theorem}{Theorem}[section]
\newtheorem{corollary}[theorem]{Corollary}
\newtheorem{lemma}[theorem]{Lemma}
\newtheorem{proposition}[theorem]{Proposition}

\theoremstyle{definition}

\theoremstyle{remark}

\newtheorem{remark}[theorem]{Remark}

\newtheorem{example}[theorem]{Example}

\newdimen\AAdi%
\newbox\AAbo%
\def\AAk#1#2{\setbox\AAbo=\hbox{#2}\AAdi=\wd\AAbo\kern#1\AAdi{}}%

\makeatletter
\def\eqlabel#1{\def\@currentlabel{#1}}

\def\formula#1{\def\@tempa{#1}\let\@tempb\theequation\def\theequation{%
\hbox{#1}}\def\@currentlabel{(\theequation)}$$}
\def\endformula{\leqno\hbox{(\@tempa)}$$\@ignoretrue\let\theequation\@tempb}

\def\given{\hskip5\p@\relax\vrule\@width.4\p@\hskip5\p@\relax}

\newcommand{\open}[1]{%
\par\normalfont\topsep6\p@\@plus6\p@\trivlist\item[\hskip\labelsep\itshape#1%
\@addpunct{.}]\ignorespaces}

\DeclareRobustCommand{\close}[1]{%
  \ifmmode 
  \else \leavevmode\unskip\penalty9999 \hbox{}\nobreak\hfill
  \fi
  \quad\hbox{$#1$}}
\makeatother

\newlength{\toskip}

\def\<{\langle}
\def\>{\rangle}

\def \Var {\textrm{Var}}
\def \Ent {\textrm{Ent}}
\def \Osc {\textrm{Osc}}

\begin{document}
\date{\today}

\title[Weak inequalities for perturbed measures]{Weak functional inequalities for perturbed measures.}

 \author[P. Cattiaux]{\textbf{\quad {Patrick} Cattiaux $^{\spadesuit}$ \, \, }}
\address{{\bf {Patrick} CATTIAUX},\\ Institut de Math\'ematiques de Toulouse. CNRS UMR 5219. \\
Universit\'e de Toulouse,
\\ 118 route
de Narbonne, F-31062 Toulouse cedex 09.} \email{patrick.cattiaux@math.univ-toulouse.fr}

 \author[P. Cordero-Encinar]{\textbf{\quad {Paula} Cordero-Encinar $^{\heartsuit}$ \, \, }}
\address{{\bf {Paula} CORDERO-ENCINAR},\\ Department of Mathematics. Imperial College London. \\
London. UK} \email{paula.cordero-encinar22@imperial.ac.uk}

\author[A. Guillin]{\textbf{\quad {Arnaud} Guillin $^{\diamondsuit}$}}
\address{{\bf {Arnaud} GUILLIN},\\Universit\'e Clermont Auvergne, CNRS, LMBP, F-63000 CLERMONT-FERRAND, FRANCE.} \email{arnaud.guillin@uca.fr}

\maketitle

 \begin{center}

 \textsc{$^{\spadesuit}$  Universit\'e de Toulouse}
\smallskip

 \textsc{$^{\heartsuit}$  Imperial College London}
\smallskip

\textsc{$^{\diamondsuit}$ Universit\'e Clermont-Auvergne}
\smallskip

\end{center}

\begin{abstract}
This paper is a follow up to an article by two of the authors dedicated to the study of Poincar\'e and logarithmic Sobolev inequalities for measures of the form $d\mu = e^{-U} d\nu$ where $e^{-U}$ is seen as a perturbation of $d\nu$. Application to the same functional inequalities for convolution products are then discussed. In the present paper we investigate similar problems for weaker functional inequalities, namely weak Poincar\'e, weighted Poincar\'e, weak log-Sobolev and weighted log-Sobolev inequalities.
\end{abstract}
\bigskip

\section{Introduction}

Functional inequalities are fundamental tools in analysis and probability: they quantify the interplay between geometry of a reference measure, stabilization rates of stochastic dynamics, and concentration/deviation properties of observables. In the reversible setting, the spectral-gap (Poincar\'e) inequality provides a direct route from a coercivity estimate on the Dirichlet form to exponential relaxation to equilibrium for the associated Markov semigroup; see, e.g., the monograph \cite{BakryGentilLedoux2014}. Yet, many measures of current interest in statistics, sampling and machine learning (heavy tails, weak confinement, multimodality, nonconvex energies) fail to satisfy a standard Poincar\'e-quality inequality with finite constant. In such regimes, one expects sub-exponential, or even sub-geometric, rates of convergence, and one must rely on weaker functional inequalities such as weak Poincar\'e or weak logarithmic Sobolev inequalities \cite{Lig,RW,CGG}.

\noindent\textbf{Perturbation viewpoint.}
A recurring theme, both classical and modern, is the \emph{stability under perturbation}:
starting from a reference probability measure $\nu$ (for which some functional inequalities are available),
we consider a perturbed target
\begin{equation}\label{eq:perturb-mu}
d\mu = {e^{-U}}\,d\nu, 
\end{equation}
and ask which inequalities (Poincar\'e, log-Sobolev, weak and weighted variants) remain valid for $\mu$,
and how the corresponding constants or \emph{rate functions} can be controlled in terms of $U$ and properties of $\nu$.
For bounded $U$, classical perturbation techniques already yield robust results (e.g.\ Holley--Stroock type arguments for log-Sobolev inequalities \cite{HolleyStroock1987}). For unbounded perturbations, the question is subtler: the effect of $U$ at infinity must be compared to the concentration/tail regime of $\nu$, and Lyapunov-type arguments become essential \cite{BCG,BBCG}.
This paper follows this perspective in the weak and weighted setting, extending the perturbation program
developed for Poincar\'e and log-Sobolev inequalities in \cite{CGPerturb}, initially motivated by a Bayesian problem,
to weak Poincar\'e, weak log-Sobolev, weighted Poincar\'e and weighted log-Sobolev inequalities,
and discussing convolution and mixture operations motivated by generative modeling.

\subsection{Weak Poincar\'e inequalities and sub-exponential relaxation}

Weak Poincar\'e inequalities (WPI) were introduced to capture convergence rates slower than exponential, by allowing a controlled defect term. A typical formulation is: there exists a non-increasing function $\beta_\mu:(0,\tfrac14]\to(0,\infty)$ such that for all sufficiently regular $f$ and all $s\in(0,\tfrac14]$,
\begin{equation}\label{eq:wpi}
\Var_\mu(f)\;\le\; \beta_\mu(s)\int |\nabla f|^2\,d\mu \;+\; s\,\Osc(f)^2,
\end{equation}
where $\Osc(f)=\sup f-\inf f$.
When $\lim_{s\downarrow 0}\beta_\mu(s)<\infty$, \eqref{eq:wpi} essentially reduces to the classical Poincar\'e inequality; otherwise, $\beta_\mu(s)\to\infty$ encodes the \emph{degree of degeneracy} of the spectral gap and yields sub-geometric rates.

The seminal paper \cite{RW} establishes that, for a large class of Markov semigroups, $L^2$-convergence rates slower than exponential are essentially equivalent to a weak Poincar\'e inequality with an appropriate rate function $\beta_\mu$. Further refinements connect the shape of $\beta_\mu$ to decay of semigroups and concentration properties \cite{WangZhang2005}, and to geometric/functional criteria (capacity, isoperimetry, tail profiles) that allow explicit computations in heavy-tailed regimes \cite{BCR1,CGGR}.
These tools are particularly powerful for distributions such as generalized Cauchy laws or Subbotin-type laws, where Poincar\'e fails but WPI holds with polynomial or logarithmic profiles, respectively \cite{CGGR}.

A key message, visible in the perturbation examples, is that meaningful perturbation results require the growth of $U$ to be compatible with the \emph{concentration scale} of the reference measure (see example \ref{exampweak2}). Roughly speaking, if $\nu$ exhibits only polynomial tails, then allowing $U$ to grow faster than logarithmically can destroy any usable control on $\beta_\mu$; conversely, if the perturbation makes tails lighter in a comparable manner, one may obtain improved rates. This ``match the tails'' principle is a recurring leitmotiv in modern perturbation theory for weak inequalities.

\subsection{Weighted Poincar\'e inequalities and weighted generators}

Weak inequalities are intimately connected to \emph{weighted} inequalities, which replace the uniform Dirichlet form by a spatially varying one. Weighted Poincar\'e-type inequalities have a long history in the study of heavy-tailed measures: even when $C_{ P}(\mu)=\infty$, one may have
\begin{equation}\label{eq:wpi-weighted}
\Var_\mu(f)\;\le\; C_{ P,\omega}(\mu)\int |\nabla f|^2\omega^2\,d\mu
\end{equation}
for an appropriate non-negative weight $\omega$. For Cauchy-type and more general convex measures, Bobkov and Ledoux \cite{BLweight} obtained sharp weighted Poincar\'e and isoperimetric-type inequalities, exhibiting canonical weights that reflect the geometry of tails. Such inequalities can be interpreted as standard Poincar\'e inequalities for a modified (non-uniform) carré du champ, and they naturally lead to dynamics whose diffusion strength depends on the position.

This motivates the study of diffusions with generator
\begin{equation}\label{eq:generator}
L^\omega f
=\omega^2 \Delta f +(\nabla\omega^2-\omega^2\nabla V)\cdot\nabla f,
\end{equation}
which can be written in divergence form as
\begin{equation}\label{eq:div}
L^\omega f = e^{V}\,\nabla\!\cdot\!\big(e^{-V}\omega^2\nabla f\big).
\end{equation}
Under standard integrability and boundary conditions, the probability measure
\begin{equation}\label{eq:inv}
d\mu(x) = Z^{-1}e^{-V(x)}\,dx
\end{equation}
is invariant and the operator is symmetric in $L^2(\mu)$, with Dirichlet form
$\mathcal{E}_\omega(f,f) = \int \omega^2|\nabla f|^2\,d\mu$.
Consequently, a weighted Poincar\'e inequality \eqref{eq:wpi-weighted} implies exponential decay of the variance along the semigroup associated with $L^\omega$ in the natural energy scale.
More generally, when only weak inequalities are available, one obtains sub-exponential relaxation rates governed by $\beta_\mu$, and one may sometimes \emph{construct} explicit weights from $\beta_\mu$ via capacitary criteria and tail rearrangements, yielding a converse weighted inequality that is optimal for the given WPI profile (see Corollary \ref{corweaktoweight}).

From an algorithmic viewpoint, $\omega$ plays the role of a \emph{preconditioner}: it changes the local diffusivity, thereby reshaping exploration of the state space. This idea resonates with recent work on optimizing or learning preconditioning metrics for Langevin dynamics and sampling (in particular, selecting diffusion coefficients to accelerate mixing while preserving a target invariant measure) \cite{cui2025optimalriemannianmetricpoincare,LelievrePavliotisRobinSantetStoltz2025OptimizingDiffusion}.
In heavy-tailed settings, such preconditioning can be essential: the unweighted Langevin diffusion may converge extremely slowly (or may not have a spectral gap), while a properly weighted dynamics can exhibit substantially improved functional-inequality structure.

\subsection{Why weak/weighted inequalities matter for diffusion-based generative modeling}

The recent success of diffusion/score-based generative models has put \emph{families of intermediate distributions} and \emph{annealed Langevin-type samplers} at the center of modern high-dimensional sampling. The classical diffusion-model pipeline defines a forward noising mechanism that progressively transforms a complex data distribution into a simple reference distribution (often Gaussian), then generates samples by simulating a reverse-time dynamics involving the \emph{score} (the gradient of the log-density) of the intermediate marginals \cite{SohlDickstein2015,HoJainAbbeel2020,song2021scorebased}.

A key theoretical and practical difficulty is that intermediate distributions need not be log-concave, and may inherit heavy tails or multi-modality from data. In this regime, the mixing properties of the Markov chains used inside samplers (e.g.\ predictor--corrector methods using Langevin \emph{corrector} steps) are not well described by classical spectral-gap/log-Sobolev theory alone. Weak and weighted functional inequalities offer a natural alternative: they provide quantitative control on stabilization when spectral gaps fail, and they can be stable under perturbations that mirror the operations used in diffusion models (smoothing/noising, reweighting, convolution, and annealing).

This connection is made explicit in the recent line of work on \emph{diffusion annealed Langevin dynamics}, where the target path of distributions is defined via Gaussian (or heavy-tailed) convolutions and sampling is performed through annealed Langevin steps; functional inequalities for intermediate targets are instrumental in proving regularity properties of scores (e.g.\ Lipschitz continuity) and deriving non-asymptotic sampling guarantees \cite{Corderoetal}.
In particular, perturbation estimates for Poincar\'e-type constants along interpolating paths provide a mechanism to control the stability of the score field and the discretization error, complementing the classical convex-logconcave sampling theory \cite{Dal17}.
Hence, understanding \emph{how weak/weighted inequalities behave under perturbations and convolutions} is not only of independent analytic interest but also directly relevant to the theoretical underpinnings of modern generative pipelines.

\subsection{Outline and contributions}

The present paper develops a perturbation theory for weak and weighted functional inequalities, organized around three complementary mechanisms:
(i) refined versions of Holley--Stroock-type perturbation arguments for weak Poincar\'e inequalities (Section \ref{secweakP}), weighted Poincar\'e inequalities (Section \ref{secweiP}), and weaker logarithmic Sobolev inequalities (Section \ref{secweakLS}),
(ii) Lyapunov-based sufficient perturbation conditions yielding explicit rate functions and stability under unbounded perturbations, for all this weaker inequalities,
and (iii) capacitary criteria linking weak Poincar\'e inequalities to (converse) weighted Poincar\'e inequalities, including explicit constructions of weights from $\beta_\mu$, (Section \ref{secWPtowP}).
We also discuss convolution products and related operations, motivated by the intermediate laws appearing in diffusion-based generative modeling, and highlight how these results interface with weighted generators of type \eqref{eq:generator} and their long-time behavior. The generalized Cauchy distributions and Subbotin (exponential power) distributions will serve as guiding examples.

\section{Perturbation of Weak Poincar\'e inequalities.}\label{secweakP}

The setting of the weak Poincar\'e inequality was introduced above, but to enhance readability, we restate the definition: $\mu$ satisfies a weak Poincar\'e inequality if there exists a non-increasing function $\beta_\mu: \mathbb R^+ \mapsto \mathbb R^+$ such that for all $s>0$ and all smooth enough function $f$, 
\begin{equation}\label{eqWPdef}
\Var_\mu(f) \leq \beta_\mu(s) \int |\nabla f|^2 d\mu + s \, \Osc^2(f) \, .
\end{equation}
 Recall that $\Var_\mu(f) \leq \Osc^2(f)/4$ so that the previous inequality is trivial for $s\geq 1/4$. If $X$ is a random variable with distribution $\mu$ we write $\beta_X=\beta_\mu$. We shall call this function the \emph{rate} of the weak inequality.

  Below we collect some properties of these weak inequalities.
\begin{proposition}\label{proppropweakP}
It holds
\begin{enumerate}
\item[(1)] \quad for all $x \in \mathbb R^d$, $\beta_{x+Z} = \beta_Z$,
\item[(2)] \quad for any $\lambda \in \mathbb R$, $\beta_{\lambda Z}=\lambda^2 \, \beta_Z$,
\item[(3)] \quad if $Z_1,...,Z_n$ are independent, $\beta_{(Z_1,...,Z_n)}(s) \leq \max_j \beta_{Z_j}(s/n)$,
\item[(4)] \quad if $Z_1$ and $Z_2$ are independent, $\beta_{Z_1+Z_2}(s) \leq \beta_{Z_1}(s/2)+\beta_{Z_2}(s/2)$,
\item[(5)] \quad if $\mu_n$ weakly converges to $\mu$, $\beta_\mu \leq \liminf \beta_{\mu_n}$.
\end{enumerate}
\end{proposition}
The proof is provided in Appendix \ref{appendix_section_2}.

Weak Poincar\'e inequalities are general enough to be satisfied by almost every probability measure, more precisely.
\begin{proposition}\label{propWPgene}
Let $\mu(dx)=e^{-V(x)} dx$ be a probability measure such that $V$ is locally bounded (for instance continuous). Then $\mu$ satisfies a weak Poincar\'e inequality where, for $d\geq 2$, $$\beta_\mu(s) \, \leq \, \frac{d+2}{d(d-1)} \, R^2(s) \, e^{h(R(s))} $$ with $$h(R)=\Osc_{B(0,R)} V \; , \; R(s)=G^{-1}(1/(1+s)) \; , \; G(u) = \mu(B(0,u)) \, .$$ For $d=1$ the pre-factor is replaced by $4/\pi^2$.
\end{proposition}
The result is an immediate consequence of Theorem 3.1 in \cite{RW}, together with the known upper bound (asymptotically sharp as $d \to +\infty$) $$C_P(\lambda_R) \leq \, \frac{d+2}{d(d-1)} R^2$$ for $\lambda_R$ the uniform probability distribution on the euclidean ball $B(0,R)$ if $d\geq 2$ (see e.g. \cite{BJM} Corollary 4.1 or \cite{CGWradial} Example 5.5). The case $d=1$ is well known. 

Actually the previous bound is not optimal in most of the specific cases that have been studied. We will only describe the situation for our two families of examples below.
\begin{example}\label{exampweak}
\begin{enumerate}
\item[]
\item[(1)] \textbf{Generalized Cauchy distribution.} 

Let $\mu(dx)= z_\alpha^{-1} (1+|x|^2)^{- (\alpha+d)/2} \, dx$ for $\alpha>0$. 

The best known result is that $\beta_\mu(s)\leq  c_\alpha \,  s^{-2/\alpha}$ for $s\leq 1/4$. It is shown in \cite{CGGR} Proposition 4.9 in the more general context of $\kappa$ concave measures (see \cite{Bobkappa} for this notion). The proof relies on weighted Poincar\'e inequalities, which we will discuss later. Notice that there is a typo (a minus is missing) in subsection 4.2.2 of \cite{BCG} where a slightly worse result is obtained ($s^{-2/\alpha'}$ for any $\alpha'<\alpha$) using weak Lyapunov-Poincar\'e inequalities introduced therein. The use of the general result in Proposition \ref{propWPgene} furnishes an intricate and in any case worse negative power. Of course the dimension dependence is hidden in the value of the constant $c_\alpha$.

We shall see below why $- 2/\alpha$ is the best possible negative power, i.e. why $s^{p} \beta_\mu(s) \to 0$ as $s \to 0$ for any $p>2/\alpha$. Notice that we recover the finiteness of $\beta_\mu(0)$ in the limit $\alpha \to + \infty$, provided we are able to control the pre-factor.
\item[]
\item[(2)] \textbf{Subbotin distributions.} 

Let $\mu(dx) = z_{\alpha} \, e^{-|x|^\alpha} dx$ for some $\alpha \in (0,1)$ (for $\alpha \geq 1$ the usual Poincar\'e inequality is satisfied).

In this case $\beta_\mu(s) \leq c_\alpha \, \ln^{\frac{2(1-\alpha)}{\alpha}}(1/s)$ for $s\leq 1/4$. This is shown (again in a more general framework replacing $|x|$ by any convex function) in Proposition 4.11 of \cite{CGGR}. The same result is obtained in \cite{BCG} subsection 4.2.1. Using Proposition \ref{propWPgene} furnishes the slightly worse exponent $4(1-\alpha)/\alpha$.
\end{enumerate}
\hfill $\diamondsuit$
\end{example}
If a Poincar\'e inequality implies exponential concentration of measure, a weak Poincar\'e inequality also implies some concentration property. The following result is part of Theorem 8 in \cite{BCR2}.
\begin{proposition}\label{propWPconcent}
If $\mu$ satisfies a weak Poincar\'e inequality with rate function $\beta_\mu$, then for all $L$-Lipschitz function $G$ with median $m_G$, $$\mu(|G-m_G|>a)\leq 6 \, \Theta(a/L)$$ where $$\Theta(u) = \inf\{s\in(0,1/4] \; ; \; u\geq 4 \, \sqrt{\beta_\mu(s)} \, \ln(1/s)\} \, .$$
\end{proposition}
Notice that in the case of Subbotin distributions the concentration function obtained with the $\beta_\mu$ in example \ref{exampweak} (2) is exactly the good one, while in the generalized Cauchy case there is an extra $\ln$, but the power function $s^{-2/\alpha}$ is the good one too.

We turn to the perturbation problem. We will only state one possible result, other ones should be obtained by directly applying Proposition \ref{propWPgene} to the perturbed measure.
\begin{theorem}\label{thmWPperturb}
Let $\mu(dx)= e^{-V(x)}dx=e^{-(U+W)(x)} dx$ for smooth $U$ and $W$ be a probability measure. Denote $\nu(dx)=e^{-W(x)}dx$ supposed to be a probability measure too. 

Assume that $U(x) \geq m_U$ (in other words $e^{-U}$ is bounded, notice that $m_U\leq 0$) and for simplicity that $m_U=U(0)$. Denote $med_\nu$ a $\nu$-median of $|x|$.

Then $$\beta_\mu(s) \leq 2 e^{\Osc_{B(0,R)} U} \beta_\nu(e^{m_U} s/7)$$  with $$R=1+med_\nu+4\sqrt{\beta_\nu(u)} \, \ln(1/u) \quad \textrm{ and } \quad u=\frac{s}{1+2\beta_\nu(s)} \, .$$
\end{theorem}
\begin{proof}
Let $R>0$ and $\chi$ be defined as $\chi(x)=\mathbf 1_{|x|\leq R-1}+(R-|x|)\mathbf 1_{R-1\leq |x|\leq R}$. We thus have for all smooth $f$ and all $a \in \mathbb R$,
\begin{eqnarray*}
\Var_\mu(f) &\leq& \int (f-a)^2 d\mu = \int (f-a)^2 \, \chi^2 \,  d\mu + \int (f-a)^2 \, (1-\chi^2) \, d\mu \\ &\leq& e^{- \min_{B(0,R)}U} \, \int (f-a)^2 \, \chi^2 \, \nu(dx) \, + \, \Osc^2 f \; \mu(|x|>R-1) 
\end{eqnarray*}
so that choosing $a$ such that $\int (f-a) \chi d\nu=0$ we have
\begin{eqnarray*} 
\Var_\mu(f) &\leq& 2 \, e^{-\min_{B(0,R)}U} \, \left(\beta_\nu(s) \int |\nabla f|^2 \, \chi^2 \, d\nu + \beta_\nu(s) \int (f-a)^2 \, |\nabla\chi|^2 \, d\nu\right) \, + \\ && \quad \quad \quad + \, \Osc^2 f \; \left(s \, e^{-\min_{B(0,R)}U} + \mu(|x|>R-1)\right)
\end{eqnarray*}
since $\Osc(g\chi) \leq \Osc(g)$. Finally, we get
\begin{eqnarray*} 
\Var_\mu(f) &\leq& 2 \, e^{\Osc_{B(0,R)}U} \, \beta_\nu(s) \int |\nabla f|^2 \,  d\mu \, + \, \\ &&  + \, \Osc^2 f \; \left(2\beta_\nu(s) \, e^{-\min_{B(0,R)}U} \,  \nu(|x|>R-1) + \,  s \, e^{-\min_{B(0,R)}U} + \mu(|x|>R-1)\right) \, \\ &\leq& 2 \, e^{\Osc_{B(0,R)}U} \, \beta_\nu(s) \int |\nabla f|^2 \,  d\mu \, + \, \\ &&  + \, \Osc^2 f \;  e^{-m_U} \, \left((2\beta_\nu(s)+1) \,  \nu(|x|>R-1) + \,  s \, \right) \, .
\end{eqnarray*}
According to Proposition \ref{propWPconcent}, if we choose $$R=1+med_\nu+4\sqrt{\beta_\nu(u)} \, \ln(1/u) \quad \textrm{ and } \quad u=\frac{s}{1+2\beta_\nu(s)}$$ then $$(2\beta_\nu(s)+1) \,  \nu(|x|>R-1) \leq 6 \, s \, .$$ The result follows.
\end{proof}
The previous result is not optimal and we could try to optimize it by choosing other $R$'s such that $(2\beta_\nu(s)+1) \,  \nu(|x|>R-1) \to 0$ as $s \to 0$. We analyze below two families of examples to determine whether this optimization really improves upon the result.

\begin{remark}\label{remWPboundperturb}
The same method furnishes the weak version of Holley-Stroock perturbation argument, i.e. if $U$ is bounded then $\beta_\mu(s) \leq e^{\Osc U} \, \beta_\nu(e^{m_U} s)$.
\hfill $\diamondsuit$
\end{remark}
\begin{example}\label{exampweak2}
\begin{enumerate}
\item[]
\item[(1)] \textbf{Generalized Cauchy distribution.} 

Let $\nu(dx)= z_\alpha^{-1} (1+|x|^2)^{- (\alpha+d)/2} \, dx$ for $\alpha>0$. We know that $\beta_\nu(s) \leq c_\alpha \, s^{-2/\alpha}$. We will not use the final statement in the Theorem but directly use the concentration result $$\nu(|x|>R-1) \leq \frac{z_\alpha^{-1}}{1+\alpha} \, (R-1)^{- \alpha}$$ for $R>1$, implying the natural choice  $R-1=s^{- (2+\alpha)/\alpha^2}$. If $U(x)-U(0) \leq c \, \psi(|x|)$ we thus have 
$$\beta_\mu(s) \leq c'_\alpha \, e^{c \psi(1+s^{-(2+\alpha)/\alpha^2})} \, s^{-2/\alpha} \, ,$$ which is  of course a disaster unless $\psi$ is equivalent to a logarithm. 
\item[]
\item[(2)] \textbf{Subbotin distribution.} 

Let $\nu(dx) = z_{\alpha} \, e^{-|x|^\alpha} dx$ for some $\alpha \in (0,1)$. This time the concentration behaves like $e^{-cR^\alpha}$ so that choosing $R$ of order $\ln^{1/\alpha}(1/s)$ we obtain  $$\beta_\mu(s) \leq c'_\alpha \, e^{c \psi(1+\ln^{1/\alpha}(1/s))} \, s^{-2/\alpha} \, ,$$ so that again, very naturally, the result is interesting provided the growth of $U$ is not larger than $|x|^\alpha$.
\end{enumerate}
Notice that when the Poincar\'e constant is finite, implying an exponential concentration, one may consider Lipschitz perturbations, i.e. $U$ with a linear growth. The takeaway of these examples is thus that in order to get interesting perturbation results the growth of the perturbation has to be at most the same as the concentration level. \hfill $\diamondsuit$
\end{example}

The natural question is then to find explicit sufficient conditions to obtain explicit rates, better than the one in Proposition \ref{propWPgene}. The first results in this direction are contained in the pioneering work \cite{RW}. In a non explicit form they are a particular case of the method based on (Foster) Lyapunov functions introduced in \cite{BCG} and inspired by the Meyn-Tweedy theory \cite{DMT}.  \cite{BCG} contains many intermediate results, as well as others weak inequalities. For the usual Poincar\'e inequality, \cite{BBCG} pushed forward the basic elements of the method, that were developed in \cite{CGopt,CGhit,CGWW,CGZ}  for others inequalities. As we will recall later, the Lyapunov method was also successfully used in the framework of weighted inequalities \cite{CGGR}. 

The Lyapunov method for weak Poincar\'e inequalities for convolution products, is developed in \cite{weakconvol}. The authors first rewrote the results in \cite{CGGR} in order to obtain the analogue of \cite{BBCG} for the Poincar\'e inequality. Namely, in their Lemma 2.4, they state the following.
\begin{theorem}\label{lyapweakP}
Let $\mu(dx)=e^{-V(x)} dx$ be a probability measure on $\mathbb R^d$, where $V$ is of class $C^1$. Denote by $L_V$ the operator $L_V:=\Delta - \nabla V. \nabla$.

Assume the following Lyapunov type condition: there exists positive constants $b$ and $R$, some positive function $\phi$ and some smooth ($C^2$) function $F$ with $F\geq 1$ such that 
\begin{equation}\label{eqweaklyap}
\frac{L_VF}{F} \leq - \, \phi + b \mathbf 1_{B(0,R)} \, .
\end{equation}
 Then $\mu$ satisfies a weak Poincar\'e inequality with rate $$\beta_\mu(s)= \left(1 + C(d) \, b \, R^2 \, e^{\Osc_{B(0,R)} V}\right) \, h_\phi^{-1}(s) \quad \textrm{ with } \quad h_\phi(r) = \mu(\phi\leq 1/r)$$  and $$C(d)= \, \frac{d+2}{d(d-1)} \textrm{ if $d\geq 2$} \quad \textrm{ and } C(1)= 4/\pi^2 \, .$$
\end{theorem}
\begin{remark}\label{remlyapR}
In the definition of the Lyapunov function it is enough to assume that $(L_VF/F)(x) \leq - \phi(x)$ for some positive $\phi$ and $|x|\geq K$ large, provided $\phi$ is continuous. Indeed if the latter is true, for $|x|\leq K$, $|L_VF/F|(x)\leq M$ and $|\phi(x)| \leq  M$ for some $M\geq 0$. It follows $$(L_VF/F)(x) \leq - \phi (x)\,\mathbf 1_{|x|>K} + M \,\mathbf 1_{|x|\leq K} \leq \, - \phi (x) + 2M\, \mathbf 1_{|x|\leq K}.$$
\hfill $\diamondsuit$
\end{remark}

\begin{proof}
The proof, given for the sake of completeness, is mimiking the proof of Theorem 1.4 in \cite{BBCG}. One has for all $a \in \mathbb R$ and all $r>0$
\begin{eqnarray*}
\Var_\mu(f) &\leq& \mu((f-a)^2) = \mu(\mathbf 1_{\phi>1/r} (f-a)^2)+\mu(\mathbf 1_{\phi\leq 1/r} (f-a)^2)\\ &\leq& \mu(r \, \phi \,  (f-a)^2) + \mu(\phi \leq 1/r) \, \Osc^2(f) \\ &\leq& - r \, \int \, \frac{LF}{F} \, (f-a)^2 \, d\mu \, + rb \, \mu((f-a)^2 \, \mathbf 1_{B(0,R)}) + \mu(\phi \leq 1/r) \, \Osc^2(f) \, .
\end{eqnarray*}
Notice that $\phi$ is measurable so that $\mu(\phi \leq 1/r):=s$ is meaningful. Using that $L$ is $\mu$-symmetric, following \cite{BBCG} p.64, one has $$- \, \int \, \frac{LF}{F} \, (f-a)^2 \, d\mu\leq \mu(|\nabla f|^2) \, .$$ 
Defining $\mu_R = (1/\mu(B(0,R))) \, \mathbf 1_{B(0,R)} \mu$ and choosing $a=\mu_R(f)$ one has $$\mu((f-a)^2 \, \mathbf 1_{B(0,R)})=\Var_{\mu_R}\left(\sqrt{\mu(B(0, R))}f\right) \leq \, e^{\Osc_{B(0,R)} V} \, C_P(\lambda_R) \, \mu(|\nabla f|^2)$$ according to Holley-Stroock perturbation argument for the Poincar\'e constant. Hence the result.
\end{proof}
\begin{example}\label{exampweak3}

For the generalized \textbf{Cauchy distribution} $\mu(dx)= z_\alpha^{-1} (1+|x|^2)^{- (\alpha+d)/2} \, dx$ for $\alpha>0$, we may choose $F(x)=(1+|x|^2)^k$. It is easily seen that, provided $k<1+(\alpha/2)$ and $R$ large enough, the Lyapunov condition is satisfied with $\phi(x)=c_\alpha/(1+|x|^2)$ for some constant $c_\alpha$. $c_\alpha$ also depends on $k$ but we may choose $k=1+(\alpha/4)$ to set the ideas. We thus recover that $h_\phi(r)$ behaves like $c \,  r^{-\alpha/2}$ and $\beta_\mu(s)$ behaves like $c \, s^{- 2/\alpha}$ for small $s$'s.

For the \textbf{Subbotin distribution} $\mu(dx) = z_{\alpha} \, e^{-|x|^\alpha} dx$ for some $\alpha \in (0,1)$, we may choose $F(x)=e^{\gamma |x|^\alpha}$ which is a Lyapunov function for $\gamma <1$ with $\phi(x)= c \, |x|^{2(\alpha-1)}$ (recall that here $\alpha<1$). It follows that $h_\phi(r) \leq c \, e^{-c \, r^{\alpha/2(1-\alpha)}}$ and finally $\beta_\mu(s)$ behaves like $c \, \ln(1/s)^{2(1-\alpha)/\alpha}$ for small $s$'s.

We thus recover in both cases the results in Example \ref{exampweak} without using weighted inequalities (even if both methods are interlocked). The courageous reader will trace the constants.

General Cauchy and Subbotin cases are studied in Example 4.3 and Example 4.2, respectively, of \cite{weakconvol} in a similar way. We provide an additional comment below.
\hfill $\diamondsuit$
\end{example}

In order to use the Lyapunov function method for a perturbation $d\mu=e^{-(U+W)}dx$, we face two problems: find a good Lyapunov function and control the tails of $\{\phi\leq 1/r\}$. The only simple Lyapunov function to test is the one of $\nu=e^{-W}dx$. This leads to the following.
\begin{theorem}\label{thmperturbWPlyap}
Let $d\nu=e^{-W} dx$. Assume that $\nu$ satisfies the Lyapunov condition \eqref{eqweaklyap}. Let $U$ be smooth and such that $d\mu=e^{-U} d\nu$ is a probability measure. Assume that $$\phi_U(x)= \phi(x) + \frac{\langle \nabla U.\nabla  F \rangle}{F}(x) >0 
$$ 
for $x$ large. Then there exists a constant $C$ such that $\mu$ satisfies a weak Poincar\'e inequality with rate $$\beta_\mu(s) = C \, h_U^{-1}(s) \quad \textrm{ where } \quad h_U(r)= e^{- \min_{\phi_U(x)\leq 1/r} U(x)} \, \nu(\phi_U\leq 1/r) \, .$$
\end{theorem}
\begin{proof}
Under the assumption of the Theorem, $F$ is still a Lyapunov function for $\mu$. It remains to evaluate $$\mu(\phi_U\leq 1/r)= \int e^{-U} \, \mathbf 1_{\phi_U\leq 1/r} \, d\nu$$ which is less than $h_U$.
\end{proof}

We shall explain this in the case of a generalized Cauchy distribution.

Let $\nu$ be a generalized Cauchy distribution of order $\alpha>0$ as in our examples, and $d\mu=e^{-U} d\nu$. For $F(x)=(1+|x|^2)^k$ to be a Lyapunov function it suffices that $c_\alpha + 2k \langle \nabla U(x), x\rangle >0$ for $x$ large enough, for the $c_\alpha$ defined in Example \ref{exampweak3}, in this case for $x$ large, $$\phi(x)=(c_\alpha + 2k \langle \nabla U(x), x\rangle)/(1+|x|^2) \, .$$ 

If $\langle\nabla U(x), x\rangle\geq c (1+|x|^2)$, $\phi$ is larger than a constant, so we know from \cite{BBCG} that $\mu$ satisfies a Poincar\'e inequality. In the spirit of the present paper (and because it was studied elsewhere) this is not an interesting situation. It is reasonable to only look at cases where $\langle\nabla U(x), x\rangle/(1+|x|^2)$ is less than a constant and $\phi$ is positive and larger than some $m/(1+|x|^2)$.

It remains to control $$h_\phi(r)=\int \, e^{-U} \, \mathbf 1_{\phi \leq 1/r} \, d\nu \leq e^{- \min_{\phi \leq 1/r} U} \, \nu(\phi \leq 1/r) \, \leq \, c \, {e^{- \min_{|x|>\sqrt{rm-1}} U(x)}} \, r^{-\alpha/2} \, .$$ 

Let us state the result we have obtained.
\begin{proposition}\label{propperturbcauchy}
Let $\nu(dx)= z_\alpha^{-1} (1+|x|^2)^{- (\alpha+d)/2} \, dx$ for $\alpha>0$, $d\mu=e^{-U}d\nu$. Define $c_\alpha$ as in Example \ref{exampweak3}. If, for $|x| \geq R$ large, $$c_\alpha + (2+(\alpha/2)) \langle \nabla U(x), x\rangle \in [a_\alpha,b_\alpha] \quad \textrm{ for some } \; a_\alpha >0 \, ,$$ $\mu$ satisfies a weak Poincar\'e inequality with a rate function 
$$\beta_\mu(s) = C(\alpha,R) \, h_\alpha^{-1}(s) \quad \textrm{ with } \quad {h_\alpha(r)= e^{- \min_{|x|>\sqrt{ra_\alpha-1}} U(x)} \, r^{-\alpha/2} }$$ 
for small $s$'s.
\end{proposition}
Notice that if $U \to + \infty$ at infinity, the previous result shows that the rate function improves upon the one of the non perturbed measure. This is quite natural since the tails of the distribution of $\mu$ are lighter so that a similar improvement is contained in the rough Proposition \ref{propWPgene}. In particular if $e^{-U}$ behaves like $(1+|x|^2)^{-\alpha'/2}$, the rate $\beta_\mu$ obtained in the previous theorem is of order $r^{- 2/(\alpha+\alpha')}$ in accordance with the rate of the generalized Cauchy distribution of order $\alpha+\alpha'$.

We include a similar result for Subbotin distributions
\begin{proposition}\label{propperturbsubbotin}
Let $\nu(dx)= z_\alpha^{-1} e^{-\vert x\vert^\alpha} \, dx$ for $\alpha\in(0,1)$, $F(x)=e^{\gamma |x|^\alpha}$ with  $\gamma <1$ be a Lyapunov function for $\nu$, and $d\mu=e^{-U}d\nu$. Define $c_\alpha$ as in Example \ref{exampweak3}. If, for $|x| \geq R$ large, $$c_\alpha + \alpha \gamma\vert x \vert^{-\alpha} \langle \nabla U(x), x\rangle \in [a_\alpha,b_\alpha] \quad \textrm{ for some } \; a_\alpha >0 \, ,$$ $\mu$ satisfies a weak Poincar\'e inequality with a rate function $$\beta_\mu(s) = C(\alpha,R) \, h_\alpha^{-1}(s) \quad \textrm{ with } \quad h_\alpha(r)= e^{- \min_{|x|>{(ra_\alpha)^{2(1-\alpha)}}} U(x)} \, e^{-c(a_\alpha r)^{\frac{\alpha}{2(1-\alpha)}}} \, ,$$ for small $s$'s.
\end{proposition}
\begin{proof}
    The proof proceeds analogously to that of the previous proposition, using the fact that $\phi>a_\alpha \vert x\vert^{^{2(\alpha -1)}}$.
\end{proof}

\begin{remark}\label{remadcom} \textbf{Additional Comment.} \quad In \cite{weakconvol} the authors also discuss the rate for convolution product. Recall that according to Proposition \ref{proppropweakP} (4), for such a convolution product $\mu*\nu$ the rate $\beta$ is (up to a constant) not worse than the worst of $\beta_\mu$ and $\beta_\nu$. In particular if $\nu$ is compactly supported on $B(0,R)$ and equivalent to Lebesgue (or uniform) measure on the ball, it satisfies some Poincar\'e inequality, so that $\beta$ is similar to $\beta_\mu$.

This result is extended in \cite{weakconvol} to the case where $\nu$, still compactly supported, is not necessarily equivalent to Lebesgue (nor absolutely continuous). The proof relies on the construction of an ad-hoc Lyapunov function and a tail control for the convolution product (similar to what we did for the pertubation in the proof of Proposition \ref{propperturbcauchy}). The latter imposes some uniform controls w.r.t. translation by $x$ (they are related to the family of probability measures $(e^{- U(x-y)}/z_x) dy$  explained in section 3 of \cite{weakconvol} to which we refer the reader). A similar problem, involving ordinary Poincar\'e inequalities, appears in our study \cite{CCG1} of the annealed Langevin dynamics.
\hfill $\diamondsuit$
\end{remark}

To (temporarily) finish with weak Poincar\'e inequalities, recall that we may replace the Oscillation by $\mathbb L^p$ norms for $p>2$,
\begin{lemma}\label{lemWPp}
Assume that $\mu$ satisfies a weak Poincar\'e inequality with rate $\beta_\mu$. Then for all $p>2$, $\mu$ satisfies a $p$-weak Poincar\'e inequality $$\Var_\mu(f) \leq \beta_\mu\left(\frac{s^{p/(p-2)}}{2^{({3}p-2)/(p-2)}}\right) \, \int |\nabla f|^2 d\mu + s \, ||f - \mu(f)||^2_{\mathbb L^p(\mu)} \, .$$
\end{lemma}
The proof is provided in the Appendix \ref{appendix_section_2}.
Due to the non optimality of the proof, for $p=+\infty$ we recover the weak Poincar\'e inequality up to a factor $4$.

\section{Weighted Poincar\'e inequalities.}\label{secweiP} 

\quad The second family of weakened Poincar\'e inequalities is the family of \emph{weighted Poincar\'e inequalities}. The terminology can be misleading. Indeed in the P.D.E. (and Analysis) community, weighted is used to characterize Poincar\'e  (or Poincar\'e-Wirtinger) inequality for measures that are not the Lebesgue measure (hence weighted measures in the absolutely continuous case). The same terminology is (unfortunately) used in \cite{BCG}. Here we shall say that $\mu$ satisfies a weighted Poincar\'e inequality with weight $\omega$, $\omega$ being a non-negative function defined on $\mathbb R^d$, if there exists a constant $C_{P,\omega}(\mu)$ such that for all nice function $f$
\begin{equation}\label{eqweP}
\Var_\mu(f) \, \leq \, C_{P,\omega}(\mu) \, \int \, |\nabla f|^2 \, \omega^2 \, d\mu \, .
\end{equation}
One can also consider the converse weighted Poincar\'e inequality 
\begin{equation}\label{eqconvweP}
\inf_a \, \int \, (g-a)^2 \; \frac{1}{\omega^2} \, d\mu \leq \, C_{cP,\omega}(\mu) \, \int \, |\nabla g|^2 \, d\mu \, .
\end{equation}

The Gaussian measure satisfies such a weighted inequality with $\omega(x)=(1+|x|^2)^{-1/2}$ and even more sophisticated weights (see \cite{Goz}). Here we are interested in weights satisfying $\omega(x) \geq 1$ for all $x$, corresponding to heavy tailed distributions when $\omega$ is not bounded from above. These inequalities have been described in \cite{BLweight,BLweight2} to study generalized Cauchy distributions and more generally, $\kappa$-concave distributions. A more general systematic study is made in \cite{CGGR}. 

From the point of view of generative models, weighted Poincar\'e inequalities can be used for Langevin dynamics, see e.g. \cite{cui2025optimalriemannianmetricpoincare, lelievre2025optimizingdiffusioncoefficientoverdamped}, even in the more general framework of non constant diffusion coefficient. 

Similar to the weak Poincar\'e inequality, the weighted Poincar\'e one satisfies the following set of properties.
\begin{proposition}\label{propweightedPIproperties}
    It holds
    \begin{enumerate}
        \item[(1)] for all $x\in\mathbb{R}^d, C_{P,\omega}(x+Z) = C_{P,\tilde\omega}(Z)$, where $\tilde \omega(z) = \omega(z-x)$.
        \item[(2)] for any $\lambda\in\mathbb{R}, C_{P,\omega}(\lambda Z) = \lambda^2C_{P,\tilde\omega}(Z)$, where $\tilde \omega(z) = \omega(z/\lambda)$.
        \item [(3)] for any map $T:\mathbb{R}^d\to\mathbb{R}^d$ which is $L$-Lipschitz and invertible, $C_{P,\omega}(T(Z)) = L^2C_{P,\tilde\omega}(Z)$, where $\tilde \omega(z) = \omega(T^{-1}(z))$.
        \item[(4)] if $Z_1, \dots, Z_n$ are independent and $P_i$ is the projection on the $i$-th coordinate, then $\Var\left(f\right)\leq \max_iC_{P,\omega_i}(Z_i)\mathbb{E}\left[\sum_i\left\vert \nabla_{Z_i} f\left( Z_1,\dots, Z_n\right)\right\vert^2\omega_i^2\circ P_i(Z_1, \dots, Z_n)\right]$.
        \item[(5)] if $Z_1, \dots, Z_n$ are independent, $\Var\left(f\left(\sum Z_i\right)\right)\leq\mathbb{E}\left[\left\vert \nabla f\left(\sum Z_i\right)\right\vert^2\sum_{i} C_{P, \omega_i}(Z_i)\omega_i^2(Z_i)\right]$.
    \end{enumerate}
\end{proposition}
The proof is provided in the Appendix \ref{appendix_section_3}.

One of the most interesting features of weighted inequalities is that they are suitable to the use of Lyapunov functions. The following result is the analogue of Theorem \ref{lyapweakP} in this framework.
\begin{theorem}\label{thmlyapweight}
Let $\mu(dx)=e^{-V(x)} dx$, for some smooth $V$, be a probability measure, $L_V=\Delta - \nabla V.\nabla$. Let $\phi$ be a $C^1$ positive increasing function defined on $\mathbb R^+$. Assume that there exists a $\phi$-Lyapunov function $F$, i.e. a $C^2$ function $F$ such that $F\geq 1$ and $$L_V F \leq - \, \phi(F) + b \, \mathbf 1_{|x|\leq R}$$ for some $R>0$ and $b \geq 0$.

Then $\mu$ satisfies both a weighted Poincar\'e inequality 
\begin{equation}\label{eqlyapwe1}
\Var_\mu(g) \leq \max \left(1 \, , \, \frac{bC_P(\mu_R)}{\phi(1)}\right) \, \int \, |\nabla g|^2 \, \left(1+ \frac{1}{\phi'(F)}\right) \, d\mu
\end{equation}
where $C_P(\mu_R)$ is the Poincar\'e constant of the normalized restriction of $\mu$ to the ball $B(0,R)$ and a converse weighted Poincar\'e inequality 
\begin{equation}\label{eqlyapwe2}
\inf_a \, \int \, (g-a)^2 \; \frac{\phi(F)}{F} \, d\mu \leq (1+bC_P(\mu_R)) \, \int \, |\nabla g|^2 \, d\mu \, .
\end{equation}
Recall that in all cases and $d\geq 2$,  $$C_P(\mu_R) \leq \frac{d+2}{d(d-1)} \, R^2 \, e^{\Osc_{B(0,R)} V}$$  while the pre-factor has to be replaced by $4/\pi^2$ for $d=1$.
\end{theorem}
These two results are exactly Theorem 2.8 and Theorem 2.18 in \cite{CGGR}. Corollary 2.14 in \cite{CGGR} contains an alternate version of \eqref{eqlyapwe1} which is not directly comparable, but is similar on most examples,
\begin{equation}\label{eqlyapwe3}
\Var_\mu(g) \leq 8 \, {\max}^2 \left(1 \, , \, \frac{bC_P(\mu_R)}{\phi(1)}\right) \, \int \, |\nabla g|^2 \, \left(1+ \frac{|\nabla F|^2}{\phi^2(F)}\right) \, d\mu \, .
\end{equation}

We now analyze how the weighted Poincar\'e behaves under perturbations.
\begin{proposition}\label{propWeightPperturb}
 Let $\mu(dx)= e^{-V(x)}dx=e^{-(U+W)(x)} dx$ for smooth $U$ and $W$ be a probability measure. Denote $\nu(dx)=e^{-W(x)}dx$ supposed to be a probability measure too that satisfies a weighted Poincar\'e inequality with weight $\omega$. 
\begin{enumerate}
     \item (Bounded perturbation) Assume that $U(x) \geq m_U$ (in other words $e^{-U}$ is bounded, notice that $m_U\leq 0$).

Then $$C_{P, \tilde\omega}(\mu)\leq e^{-m_U} C_{P, \omega}(\nu)\quad\textrm{with} \quad\tilde\omega=\omega e^{U/2} \, .$$
Furthermore, if $U$ is bounded, then $\mu$ satisfies a weighted Poincar\'e inequality with weight $\omega$ and constant $$C_{P, \omega}(\mu)\leq e^{\Osc(U)} C_{P, \omega}(\nu)\, .$$
\item (Unbounded perturbation)
\begin{enumerate}
    \item (Weighted Lipschitz case) Suppose that there exists $\varepsilon>0$ such that
    $$s=\sup_x C_{P, \omega}(\nu)\frac{1+\varepsilon}{4}|\nabla U|^2 \, \omega^2<1$$
    then $\mu$ satisfies a weighted Poincar\'e inequality with weight $\omega$ and constant
    $$C_{P,\omega}(\mu)\le \frac{(1+\varepsilon^{-1})C_{P,\omega}(\nu)}{1-s}$$
    \item (Weighted generator case) Let us introduce
    $$L_{U+W}^{\omega}f = \omega^2 \Delta f +(\nabla \omega^2-\omega^2 \nabla (U+ W)) .\nabla f .$$
    Suppose that
    $$s=\sup_x \frac{1}{2}C_{P,\omega}(\nu)\left(\frac12|\nabla U|^2\omega^2+L_{U+W}^{\omega}U\right)_+<1$$
    then $\mu$ satisfies a weighted Poincar\'e inequality with weight $\omega$ and constant
    $$C_{P,\omega}(\mu)\le \frac{C_{P,\omega}(\nu)}{1-s}$$
    \item (Weighted Lyapunov condition) Let us suppose that there exists $F\ge 1$, some $R>0$ and constants $b,\theta>0$ such that
    \begin{equation}\label{lyapw}
    L_W^\omega F\le -\theta F+b1_{|x|\le R}
    \end{equation}
    and there exists $\theta'<\theta$
    $$-\omega^2\nabla U.\nabla F\le \theta' F$$
    then $\mu$ satisfies a weighted Poincar\'e inequality with weight $\omega$ and constant
    $$C_{P,\omega}(\mu)\le \frac{1}{\theta-\theta'}(1+C_{P,\omega}(\mu_R)).$$
\end{enumerate}
\end{enumerate}
\end{proposition}
The proof is provided in the Appendix \ref{appendix_section_3}.

Besides, consider $g=(f-a)/\omega$. Assume that $\int (f-a)/\omega \,\,  d\mu=0$, then the weighted Poincar\'e inequality yields, for $\lambda>1$
\begin{eqnarray*}
\int \, \left(\frac{f-a}{\omega}\right)^2 \, d\mu &=& \Var_\mu(g) \leq C_{P,\omega}(\mu) \, \int \, |\nabla f   -  \, (\nabla \omega/\omega) \, (f-a)|^2  \, d\mu \\ &\leq& C_{P,\omega}(\mu) \left(\lambda \int \, |\nabla f|^2 d\mu \, + \, \frac{\lambda}{\lambda-1} \, \int (f-a)^2 \, |(\nabla \omega/\omega)|^2  \, d\mu\right) 
\end{eqnarray*}
so that, provided $C_{P,\omega}(\mu) \, |\nabla \omega|^2 <1$ we obtain the existence of some constant $C_{cP,\omega}(\mu)$ such that 
\begin{equation*}
\inf_a \int \, (f-a)^2 \, \frac{1}{\omega^2} \, d\mu \leq C_{cP,\omega}(\mu) \, \int \, |\nabla f|^2 d\mu \, .
\end{equation*}
The latter \emph{converse weighted inequality} has been discussed in \cite{DMC,Adrienetal} under the name Hardy/ Poincar\'e inequalitiy (the latter terminology unfortunately also covers different inequalities) in order to study fast diffusion equations. It is easily seen that, conversely, \eqref{eqconvweP} implies \eqref{eqweP} under the same assumption on $\omega$.

The relationship between (converse) weighted Poincar\'e inequalities and weak Poincar\'e inequalities is described below (see Theorem 4.6 in \cite{CGGR}).
\begin{theorem}\label{thmWPtoweP}
If $\mu$ satisfies the converse weighted Poincar\'e inequality \eqref{eqconvweP} and $\int (1/\omega^2) d\mu < +\infty$ then, $\mu$ satisfies a weak Poincar\'e inequality $$\Var_\mu(f) \leq \frac{C_{cP,\omega}(\mu)}{G(s)} \, \int |\nabla f|^2 \, d\mu \, + \, s \, \Osc^2(f)$$ with $G(s)=\inf\{u \; ; \; \mu(\omega^2 \geq \frac 1u)>s \}$ for $s<\frac 14$.
\end{theorem}

Let us start by giving examples of weighted Poincar\'e inequalities for our favorite two classes of measures: generalized Cauchy and Subbotin types. Then we will consider perturbation examples of these classes.
\begin{example}\label{exwPcauchy}
 \textbf{Generalized Cauchy distributions.}  

Let $\mu_\alpha(dx)=z_\alpha^{-1} (1+|x|^2)^{-(d+\alpha)/2}$ with $z_\alpha= \frac{\Gamma(\alpha/2) \, \pi^{d/2}}{\Gamma((\alpha+d)/2)}$ and $\alpha>0$.

This family of examples has been the most studied. As we already mentioned, converse weighted inequalities appeared in \cite{DMC,Adrienetal} as a tool for studying non linear diffusions. It is shown in \cite{Adrienetal} that $\omega(x)=\sqrt{1+|x|^2}$ is the good weight, and that $C_{cP,\omega}(\mu_\alpha)\leq 1/(\alpha+d)$ if $\alpha \geq d \geq 3$, while the obtained bound for $C_{cP,\omega}(\mu_\alpha)$ when $\alpha <d$ depends on the dimension $d$.

Later in \cite{BLweight}, using some generalization of the Brascamp-Lieb inequality, the authors derived both a weighted and a converse weighted inequality for large $\alpha$. Namely, it is shown in \cite{BLweight} Theorem 3.1 and Corollary 3.2 that
\begin{proposition}\label{propweightpoinccauchy}
If $\alpha>d$, $\mu_\alpha$ satisfies a weighted Poincar\'e inequality with optimal weight $\omega(x)=\sqrt{1+|x|^2}$ and constant $C_{P,\omega}(\mu_\alpha) \leq 2/(\alpha+d-2)$ (actually $(1+o(1))/(d+\alpha)$ where $o(1) \to 0$ as $\alpha \to +\infty$, see the remark below Theorem 3.1 in \cite{BLweight}).

If $\alpha \geq d+2$, $\mu_\alpha$ satisfies a converse weighted Poincar\'e inequality with optimal weight $\omega(x)=\sqrt{1+|x|^2}$ and constant $C_{cP,\omega}(\mu_\alpha)\leq 1/(\alpha+d)$.
\end{proposition}
Notice that the range for $\alpha$ in the converse weighted inequality is smaller than in \cite{Adrienetal}. The weight is optimal thanks to the correspondence with weak inequalities and concentration.

Also notice that, since $|\nabla \omega|<1$ one can use the change of functions we previously described and get that for $1<\lambda<1/C_{P,\omega}(\mu_\alpha)$, $$C_{cP,\omega}(\mu_\alpha)\leq \frac{\lambda \, C_{P,\omega}(\mu_\alpha)}{(\lambda -1)(1-\lambda C_{P,\omega}(\mu_\alpha))} \, .$$ Optimizing in $\lambda$, provided $C_{P,\omega}(\mu_\alpha) < 1$, we obtain for $\lambda=1/\sqrt{C_{P,\omega}(\mu_\alpha)}$, 
\begin{equation}\label{eqcaucwptrick}
C_{cP,\omega}(\mu_\alpha) \, \leq \, C_{P,\omega}(\mu_\alpha)/(1-\sqrt{C_{P,\omega}(\mu_\alpha)})^2 \, .
\end{equation}
This bound is also derived in Proposition 3.3 of \cite{BLweight}.

Finally in the very recent \cite{Hug}, a complete description of weighted inequalities is obtained. The following is Theorem 1.1 in \cite{Hug}.
\begin{theorem}\label{thmweightpoinccauchy}
$\mu_\alpha$ always satisfies some weighted Poincar\'e inequality with optimal weight $\omega(x)=\sqrt{1+|x|^2}$ and constant $C_{P,\omega}(\mu_\alpha)$ less than
\begin{enumerate}
\item[(1)] \quad for $d=1$, $4/\alpha^2$ if $0<\alpha\leq 2$; $1/(\alpha-1)$ if $\alpha\geq 2$,
\item[(2)] \quad for $d \geq 2$, $4/\alpha^2$ if $0<\alpha\leq 4$; $1/2(\alpha-2)$ if $4\leq \alpha\leq d+2$; $1/(\alpha+d-2)$ if $\alpha \geq d+2$.
\end{enumerate}
\end{theorem}
Notice that for large $\alpha$ we recover the first part of Proposition \ref{propweightpoinccauchy}, but using \eqref{eqcaucwptrick} is far from providing us with the converse inequality in Proposition \ref{propweightpoinccauchy}. Nevertheless, \eqref{eqcaucwptrick} furnishes some bound for $C_{cP,\omega}$ as soon as $\alpha > 2$.

In \cite{CGGR} this example (in the more general setting of $\kappa$-concave distributions, also studied in \cite{BLweight2}) is studied as an application of Theorem \ref{thmlyapweight}. However the study of the pre-factor is not done. We shall below carefully (but briefly) follow the calculations of \cite{CGGR} in order to give a precise estimate on the constants and thus will fill the gap $0<\alpha \leq 2$.

Choose $F(x)=(1+|x|^2)^{(k/2)+1}$ for some $\alpha>k>0$. Then $$L_VF=(k+2) \, (1+|x|^2)^{(k/2)-1} \, (\alpha+d-k \, - \, (\alpha-k)(1+|x|^2))$$ so that choosing 
\begin{equation}\label{eqR}
R=\sqrt{\frac{d+\varepsilon}{\alpha-k - \varepsilon}} \; \textrm{ for some $\varepsilon< \alpha-k$ },
\end{equation}
for $|x|\geq R$ it holds $$L_V F \leq - \, (k+2)\varepsilon \,  F^{k/(k+2)} \, .$$ It follows that $\phi(u)= \varepsilon (k+2) \, u^{k/(k+2)}$ so that $\phi'(u)= \varepsilon k u^{-2/(k+2)}$ and $$\frac{1}{\phi'(F(x))}= \frac{1}{\varepsilon \, k} \, (1+|x|^2) \, ,$$ so that $$(1+(1/\phi'(F(x)))\leq \frac{\varepsilon k + 1}{\varepsilon k} \, (1+|x|^2) \, . $$ Except for the pre-factor the result does not depend on $k$. Thus  we can always choose $k\leq 2$ so that for all $x$, $(1+|x|^2)^{(k/2)-1} \leq 1$.  Finally, we have obtained for $\varepsilon < \alpha-k$, 
$$L_V F \leq  \, -  \, \varepsilon \, (k+2)  \, F^{k/(k+2)} + (k+2)\left((d+\alpha-k)+\varepsilon (1+R^2)^{{k/2}}\right)  \, \mathbf 1_{|x|\leq R} \, .$$
We have thus obtained, for $0<k<\min(2,\alpha)$, $$\Var_{\mu_\alpha}(g) \leq  \, \max(1,C) \left(\frac{1+\varepsilon k}{\varepsilon k}\right) \, \int \, |\nabla g|^2 \, (1+|x|^2) \, d\mu_\alpha $$ with 
$$C= \frac{(d+\alpha-k)+\varepsilon (1+R^2)^{{k/2}}}{\varepsilon} \; C_P(\mu_{\alpha,R}) \, .$$
If one gets the optimal weight, the constant is very bad, and still worse when evaluating $C_P(\mu_{\alpha,R})$. However we also obtain, in all cases, the following converse inequality, for  $0<k<\min(2,\alpha)$,
\begin{equation}\label{eqcauchyconverse}
\inf_a \, \int \frac{(g-a)^2}{1+|x|^2} \, \mu_\alpha(dx) \, \leq \, \frac{C}{\varepsilon} \, \int \, |\nabla g|^2 \, \mu_\alpha(dx)
\end{equation}
where $$C= 1/{(k+2)} + ((d+\alpha-k)+\varepsilon (1+R^2)^{{k/2}}) \; \frac{d+2}{d(d-1)} \, R^2 (1+R^2)^{(\alpha+d)/2} \, ,$$ for $d\geq 2$ and $R$ defined in \eqref{eqR}. For $d=1$ the factor $(d+2)/(d(d-1))$ is replaced by $4/\pi^2$.
This result covers the missing case $0<\alpha\leq 2$. 

The method of proof in \cite{Hug} consists in observing that $\mu_\alpha$ is symmetric for the weighted operator $$L^\omega_\alpha f = \omega^2 \Delta f +\left(\nabla \omega^2-\omega^2 \, \frac{(d+\alpha)x}{1+|x|^2}\right) .\nabla f \, .$$  The associated carr\'e du champ is thus $\omega^2 \, |\nabla f|^2$. It turns out that this operator satisfies the $CD(0,+\infty)$ curvature condition (see \cite{BaGLbook}) but no $CD(\rho,+\infty)$ condition for $\rho>0$. This is shown in Proposition 4.3 of \cite{Hug}. Huguet thus replaces the usual Bakry-Emery criterion by the integrated criterion in order to prove a usual Poincar\'e inequality with this new carr\'e du champ, i.e. a weighted Poincar\'e inequality. \hfill $\diamondsuit$
\end{example}

\begin{remark}\textbf{Convolution of generalized Cauchy distributions}

Unlike the Gaussian family the generalized Cauchy family is not stable, making the analysis of convolutions more complex. \cite{NAD} provides explicit expressions for the convolution in dimension $d=1$ when both $\alpha_1$ and $\alpha_2$ are odd, and the scale parameter is equal to $1$. In particular, in this situation it holds that
    $$
    \mu_{\alpha_1}*\mu_{\alpha_2}(x) = \mathcal{O}\left((1+\vert x\vert^2)^{-{(1+\alpha_{\min})/2}}\right), \quad \alpha_{\min}=\min\{\alpha_1, \alpha_2\}.
    $$
More generally, denote $F(x)= \frac{\mu_{\alpha_1}*\mu_{\alpha_2}(x)}{\mu_{\alpha_{\min}}(x)}$. We may write 
\begin{eqnarray*}
F(x) &=& C(\alpha_1,\alpha_2,d) \, \int \, \frac{(1+|x|^2)^{(d+\alpha_{min})/2}}{(1+|x-y|^2)^{(d+\alpha_{min})/2} \, (1+|y|^2)^{(d+\alpha_{max})/2}} \, dy \\ &:=& C(\alpha_1,\alpha_2,d) \, \int G_{d,\alpha_{min},\alpha_{max}}(x,y) \, dy \, .
\end{eqnarray*}
$F$ is continuous and everywhere positive. In addition, on one hand 
\begin{eqnarray*}
G_{d,\alpha_{min},\alpha_{max}}(x,y) &\leq & \frac{(1+2|x-y|^2+2|y|^2)^{(d+\alpha_{min})/2}}{(1+|x-y|^2)^{(d+\alpha_{min})/2} \, (1+|y|^2)^{(d+\alpha_{max})/2}}\\ &\leq& c(m) \, \frac{(1+|x-y|^2)^{(d+\alpha_{min})/2} + (1+|y|^2)^{(d+\alpha_{min})/2}}{(1+|x-y|^2)^{(d+\alpha_{min})/2} \, (1+|y|^2)^{(d+\alpha_{max})/2}}\\ &\leq& \frac{c(m)}{(1+|y|^2)^{(d+\alpha_{max})/2}} + \frac{c(m)}{(1+|x-y|^2)^{(d+\alpha_{min})/2} \, (1+|y|^2)^{(\alpha_{max}-\alpha_{min})/2}} \, \\ &\leq& \frac{c(m)}{(1+|y|^2)^{(d+\alpha_{max})/2}} + \frac{c(m)}{(1+|x-y|^2)^{(d+\alpha_{min})/2}} \, ,
\end{eqnarray*}
so that $F$ is upper bounded. 

On the other hand, $$F(x) \geq C(\alpha_1,\alpha_2,d) \, \int \, \mathbf 1_{|y|\leq 1} \, G_{d,\alpha_{min},\alpha_{max}}(x,y) \, dy$$ and for $|x|>2$, $|x-y| \mathbf 1_{|y|\leq 1} \leq |x|+1$. Hence, for $|x|>2$, 
$$F(x) \geq C(\alpha_1,\alpha_2,d) \, \int\, \mathbf 1_{|y|\leq 1} \, \frac{(1+|x|^2)^{(d+\alpha_{min})/2}}{(2+|x|^2)^{(d+\alpha_{min})/2} \, (1+|y|^2)^{(d+\alpha_{max})/2}} \, dy$$ which is bounded below by a positive constant. For $|x|\leq 2$ one also have a positive lower bound thanks to the continuity of $F$.

Hence $\mu_{\alpha_1}*\mu_{\alpha_2}$ is a bounded perturbation $F \, \mu_{\alpha_{min}}$ where $F$ is bounded from below and from above by positive constants. According to Proposition \ref{propWeightPperturb} (1), $\mu_{\alpha_1}*\mu_{\alpha_2}$ satisfies a weighted Poincar\'e inequality with the same optimal weight $\omega(x)=\sqrt{1+|x|^2}$.
\hfill $\diamondsuit$
\end{remark}

More generally one can ask about the weight for a convolution product $\mu*\nu$. As for Remark \ref{remadcom} we think that, if $\nu$ is compactly supported, the weight is the same as for $\mu$. We will not prove this claim. The proof is presumably similar to the case of weak inequalities studied in \cite{weakconvol} simply adapting the construction of the Lyapunov function.

\begin{example}\label{exwPsubbotin}
 \textbf{Subbotin distributions.}  

This case, (2) in Example \ref{exampweak2}, is less understood. The only known result seems to be Proposition 3.6 in \cite{CGGR}, where the Lyapunov function method is used. We will not give all the details, in particular some explicit (but of course not satisfying) bounds for the constants. The optimal weight in this case is $\omega^2(x)=1+(1+|x|)^{2(1-\alpha)}$.
\hfill $\diamondsuit$
\end{example}

Following the results in Proposition \ref{propWeightPperturb}, we analyze the behavior of weighted Poincar\'e inequalities under perturbations for generalized Cauchy and Subbotin distributions. We focus on unbounded perturbations.
\begin{example}\label{exwPcauchyperturbed}
 \textbf{Perturbations of generalized Cauchy distributions.} 
Recall that the weight function is of the form $\omega^2(x) = 1+\vert x\vert^2$.
\begin{enumerate}
    \item[(a)] (Weighted Lipschitz case) The perturbation needs to satisfy
    \begin{equation*}
        \vert \nabla U\vert^2\leq \frac{4}{(1+\varepsilon)\omega^2C_{P, \omega}(\nu)}.
    \end{equation*}
    Therefore, the ``maximal'' perturbation satisfying this is of the form
    \begin{equation}\label{eqweightedPperturbcauchylike}
        U(x)-U(0) = \beta \ln(1+\vert x\vert),\,\,\, \text{where}\,\,\,\beta^2=\frac{4}{(1+\varepsilon)C_{P, \omega}(\nu)}.     
    \end{equation}
    \item[(b)] (Weighted generator case) First, observe that
    $$\frac12|\nabla U|\omega^2+L^{\omega,\mu}U=\omega^2\nabla^2U+\langle\nabla\omega^2,\nabla U\rangle-\frac{1}{2}\omega^2\vert \nabla U\vert^2-\omega^2\langle \nabla W, \nabla U\rangle.$$
    The condition
    \begin{equation*}
        \omega^2\nabla^2U+\langle\nabla\omega^2,\nabla U\rangle\leq C
    \end{equation*}
    is equivalent to $\omega^2\nabla U$ being Lipschitz continuous.
    Note that potentials of the form \eqref{eqweightedPperturbcauchylike} also satisfy the following condition
\begin{equation}\label{eqconditionweightedlyapcaseperturb}
        s=\sup_x \frac{1}{2}C_{P,\omega}(\nu)\left(\frac12|\nabla U|\omega^2+L^{\omega,\mu}U\right)_+<1.
    \end{equation}
    \item[(c)] (Weighted Lyapunov condition) Suppose that $\beta=(\alpha+d)/2>1$. We first show that there exists a Lyapunov function $F\geq 1$ and $R, b, \theta>0$ such that
    $$L_W^\omega F:= \omega^2\Delta F-2(\beta -1)\langle x, \nabla F\rangle\leq -\theta F + b1_{\vert x\vert\leq R}.$$
    Consider $F = 1+\vert x\vert^2 = \omega^2$. Then,
    \begin{align*}
        L_W^\omega F= 2\omega^2 -4(\beta-1)\vert x\vert^2 = 2\omega^2-4(\beta-1)\omega^2\left(1-\frac{1}{\omega^2}\right).
    \end{align*}
    Choose $R$ such that $1-(1+\vert R-1\vert^2)^{-1} = (2(\beta-1))^{-1}$. For $|x|>R$, it follows that 
    \begin{equation*}
         L_W^\omega F\leq -\theta\omega^2,
    \end{equation*}
    with $\theta>0$. {Taking $b = (2+\theta)\omega^2(R)$}, we obtain the weighted Lyapunov condition \eqref{lyapw}. For a perturbation to be valid, there must exist $\theta'<\theta$ such that $-\omega^2\nabla U.\nabla F\le \theta' F$, which is equivalent to 
    $$-\langle x, \nabla U\rangle\leq \frac{\theta'}{2}.$$
    This condition is satisfied for potentials of the form \eqref{eqweightedPperturbcauchylike}. In this case, however, one has $\beta = \theta'/2$. 
\end{enumerate}
\hfill $\diamondsuit$
\end{example}
Similar perturbation results for Subbotin distributions are deferred to Appendix \ref{appendix_section_3}.

\begin{remark}\label{remcompare}
One can compare the perturbations of this section with Theorem \ref{thmperturbWPlyap} and Proposition \ref{propperturbcauchy} in the previous one. Roughly speaking, when one can compare all the involved inequalities (as in the Cauchy case), all results are similar. However, since the devil is in the detail, one can also easily see some differences, and fortunately enough, sometimes, some complementary results.
\hfill $\diamondsuit$
\end{remark}

\section{More on the links between weak and weighted Poincar\'e inequalities.} \label{secWPtowP}

We already saw in Theorem \ref{thmWPtoweP} that a converse weighted Poincar\'e inequality together with a tail control of the considered measure, implies a weak Poincar\'e inequality. In some situations, for instance for the generalized Cauchy distributions, weighted and converse weighted Poincar\'e inequalities can be compared. 

We shall here look at the converse way, i.e. from weak Poincar\'e to converse weighted Poincar\'e inequalities.

Introduce the standard (global) $L^2(\mu)$-capacity of a set:
\begin{equation}\label{eq:cap}
  Cap_\mu(A):=\inf\left\{ \int |\nabla \varphi|^2\,d\mu:\ \varphi\in\mathrm{Lip}_c(\mathbb R^d),\ 1_A\le \varphi\le 1\right\}.
\end{equation}
Capacity measure criteria for functional inequalities were introduced in \cite{BRstudia} based on results  in \cite{Masobolev}. They were further developed in \cite{BCR1,BCR2,BCR3,CGG} among others.

We assume that $\mu$ satisfies some weak Poincar\'e inequality \eqref{eqWPdef}, and that $\lim_{s\downarrow 0}\beta_\mu(s)=+\infty$, otherwise an ordinary Poincar\'e inequality is satisfied. Let us  state a first result (Theorem 2.1 in  \cite{BCR2}).
\begin{lemma}[Measure--capacity bound for weak Poincar\'e]\label{prop:cap-lb}
Assume \eqref{eqWPdef}. Then for every measurable set $A\subset\mathbb R^d$ with $\mu(A)=a\le\frac12$,
\begin{equation}\label{eq:cap-lower}
  Cap_\mu(A)\ \ge\ \frac{a}{4\,\beta_\mu(a/4)} \, .
\end{equation}
\end{lemma}
We have the following capacity measure to converse weighted Poincar\'e inequality. 
\begin{lemma}[Capacitary criterion for converse weighted Poincar\'e]\label{lem:cap-to-weighted}
Let $\omega\ge0$ be measurable and let $\nu=\omega^2 \, \mu$.
Assume there exists $C\in(0,\infty)$ such that
\begin{equation}\label{eq:nuCap}
  \int_A \omega^2 \, d\mu\ =:\ \nu(A)\ \le\ C \, Cap_\mu(A)\qquad\text{for every Borel }A\subset\mathbb R^d\text{ with }\mu(A)\le\frac12.
\end{equation}
Then for every locally Lipschitz $f$ and any median $m$ of $f$ (so that $\mu(f\ge m)\ge\tfrac12$ and $\mu(f\le m)\ge\tfrac12$),
\begin{equation}\label{eq:weightedPoincareC}
  \int (f-m)^2\,\omega^2 \,d\mu \ \le\ C'\,\int |\nabla f|^2\,d\mu.
\end{equation}
In particular,
\(
\inf_{c\in\mathbb R}\int (f-c)^2\,\omega^2 \, d\mu \le C'\,\int |\nabla f|^2\,d\mu.
\)
\end{lemma}
The proof is given in Appendix \ref{appendix_section_4}.

We shall use these two lemmas in order to build an appropriate weight.

Let $x_0\in\mathbb R^d$ be fixed and define the tail function
\begin{equation}\label{eq:tail}
  s(r):=\mu\{x\in\mathbb R^d:\ |x-x_0|>r\},\qquad r\ge0.
\end{equation}
Note that $s$ is non-increasing and $s(r)\downarrow0$ as $r\uparrow\infty$. Also define
\begin{equation}\label{eq:h-def}
  h(u):=\frac{1}{4\,\beta_\mu(u/4)},\qquad u\in(0,1].
\end{equation}

\begin{proposition}[Explicit weight]\label{prop:explicit-weight}
Define
\begin{equation}\label{eq:omega-explicit}
  \omega^2(x)\ :=\ h\bigl(s(|x-x_0|)\bigr)
  \ =\ \frac{1}{4\,\beta_\mu\!\left(\tfrac14\,\mu\{y:\ |y-x_0|>|x-x_0|\}\right)}.
\end{equation}
Then $\omega^2$ is radial, non-increasing with respect to $|x-x_0|$, and $\omega(x)\to0$ as $|x|\to\infty$.
Moreover, for every Borel set $A$ with $\mu(A)\le\tfrac12$,
\begin{equation}\label{eq:nu-cap-small}
  \int_A \omega^2\,d\mu \ \le Cap_\mu(A).
\end{equation}
\end{proposition}

\begin{proof}
The monotonicity and vanishing at infinity are immediate from the definition and the fact that $s(r)\downarrow0$ whereas $\beta_\mu(u)\uparrow\infty$ as $u\downarrow0$.

To prove \eqref{eq:nu-cap-small}, we compare $\int_A \omega^2\,d\mu$ with $Cap(A)$ via rearrangements.
Let $\omega^*$ denote the non-increasing rearrangement of $\omega^2$ with respect to $\mu$.
A standard calculation shows that for $h$ non-increasing, the choice $\omega^2(x)=h(s(|x-x_0|))$ gives
\begin{equation}\label{eq:rearrangement}
  \omega^*(u)=h(u)\qquad\text{for all }u\in(0,1). 
\end{equation}
Indeed,
$$
\{\omega^2>t\}=\{h(s(|x-x_0|))>t\}=\{s(|x-x_0|)<h^{-1}(t)\}
$$
so that
$$
\mu(\{\omega^2>t\})=h^{-1}(t)
$$
(by definition of the generalized inverse of a monotone map), whence \eqref{eq:rearrangement}.
Therefore, for any $A$ with $\mu(A)=a\le\tfrac12$,
\begin{equation}\label{eq:nuA-via-omega-star}
  \int_A \omega^2\,d\mu \ \le\ \int_0^{a} \omega^*(u)\,du \ =\ \int_0^{a} h(u)\,du
  \ =\ \int_0^{a} \frac{du}{4\,\beta_\mu(u/4)}.
\end{equation}
Since $\beta_\mu$ is non-increasing, $u\mapsto 1/\beta_\mu(u/4)$ is non-decreasing, so that
$$
\int_0^{a} \frac{du}{\beta_\mu(u/4)} \le \frac{a}{\beta_\mu(a/4)}.
$$
Combining this with \eqref{eq:nuA-via-omega-star} and the capacity lower bound \eqref{eq:cap-lower} yields
$$
\int_A \omega^2\,d\mu \le \frac{a}{4\,\beta_\mu(a/4)} \le \, Cap_\mu(A) \, ,
$$ which is \eqref{eq:nu-cap-small}.
\end{proof}

\begin{corollary}\label{corweaktoweight}
Assume that $\mu$ satisfies a weak Poincar\'e inequality \eqref{eqWPdef}, with $\lim_{s\downarrow 0}\beta_\mu(s)=+\infty$. 
Then $\mu$ satisfies a converse weighted Poincar\'e inequality with weight $\omega$ defined by \eqref{eq:omega-explicit}.
\end{corollary}

\section{Weaker Logarithmic Sobolev inequalities.}\label{secweakLS}

Perturbation of logarithmic Sobolev are studied in \cite{CGPerturb} Theorems 2.5, 2.6 and 2.7. Lyapunov conditions for log-Sobolev (for short) inequalities are introduced in \cite{CGWu} in the more general framework of (weighted) super Poincar\'e inequalities and discussed in \cite{CGhit} (with an erratum on some missing assumptions, one can find on the webpage of the first author). Using the Lyapunov method of \cite{CGWu}, \cite{WW} studies the log-Sobolev inequality for convolution products $\mu*\nu$. Of course if both $\mu$ and $\nu$ satisfy a log-Sobolev inequality, so does $\mu*\nu$ with a log-Sobolev constant less than the sum of the two. \cite{WW} contains more general results, in the spirit of \cite{weakconvol}, in particular the authors show that if $\mu$ satisfies a log-Sobolev inequality and $\nu$ is compactly supported, then $\mu*\nu$ also satisfies a log-Sobolev inequality.

Similarly to weakened Poincar\'e inequalities one may consider weaker logarithmic Sobolev inequalities.

First, the \emph{weak logarithmic Sobolev inequality}, introduced in \cite{CGG}, states that there exists a non-increasing function $\beta_\mu:\mathbb{R}^+\mapsto\mathbb{R}^+$ such that for all $s>0$,
\begin{equation}\label{eqweakLS}
    \int f^2 \, \ln\left(\frac{f^2}{\mu(f^2)}\right) \, d\mu := \Ent_{\mu}(f^2)\leq \beta^{LS}_\mu(s) \int |\nabla f|^2 d\mu + s \, \Osc^2(f) \, .
\end{equation}
As shown in section 4 of \cite{CGG} these inequalities allow to characterize non exponential convergence in relative entropy. The following set of properties hold.
\begin{proposition}\label{propertieswLSI}
    It holds (we skip the superscript $LS$ for simplicity)
    \begin{enumerate}
        \item[(1)] for all $x\in\mathbb{R}^d, \beta_{x+Z} = \beta_Z$,
        \item[(2)] for any $\lambda\in\mathbb{R}, \beta_{\lambda Z} = \lambda^2\beta_Z$,
        \item[(3)] if $Z_1, \dots, Z_n$ are independent, $\beta_{(Z_1, \dots, Z_n)}(s)\leq \max_i\beta_{Z_i}(s/n)$,
        \item[(4)] if $Z_1, \dots, Z_n$ are independent, $\beta_{\sum_i Z_i}(s)\leq \sum_i\beta_{Z_i}(s/n)$,
        \item[(5)] if $\mu_n$ weakly converges to $\mu$, $\beta_{\mu} \leq\liminf\beta_{\mu_n}$.
    \end{enumerate}
\end{proposition}
The proof is provided in Appendix \ref{appendix_section_5}.

The main fact is that actually, a weak log-Sobolev inequality is equivalent to a weak Poincar\'e inequality. This is Proposition 3.1 in \cite{CGG}.
\begin{proposition}\label{propcgg1}
If $\mu$ satisfies a weak log-Sobolev inequality with rate $\beta_\mu^{LS}$, it satisfies a weak Poincar\'e inequality with 
\begin{equation}\label{eqwlswp}
\beta_\mu(s) \leq \frac{24 \, \beta_\mu^{LS}\left(\frac s2 \, \ln(1+(1/2s)\right)}{\ln (1+(1/2s))} \, ,
\end{equation}
provided $\beta_\mu$ in \eqref{eqwlswp} is non-increasing.

Conversely for some positive constants $c,c',s_0$ and all $s \leq s_0$, 
\begin{equation}\label{eqwpwls}
\beta_\mu^{LS}(s) = c' \, \beta_\mu\left(\frac{cs}{\ln (1/s)}\right) \, \ln(1/s) \, . 
\end{equation}
\end{proposition}
The constants $c,c',s_0$ are universal and can be traced in the proofs of \cite{CGG} if necessary. Notice that the previous result is almost closed. Indeed, if we plug the formula \eqref{eqwpwls} for $\beta_\mu^{LS}$ in equation \eqref{eqwlswp}, we get something like $\beta_\mu(s) \leq C \, \beta_\mu( Cs)$, i.e. up to some constants, we are not loosing a factor for $\beta_\mu$. 

The fact that $\beta_\mu$ in \eqref{eqwlswp} is non-increasing is crucial (and not explicitly stated in \cite{CGG}, but necessary in order to apply Theorem 2.2 in \cite{BCR2}). Indeed if $\beta_\mu^{LS}=C_{LS}$ is constant, the result does not apply (otherwise $C_P(\mu)=0$). Except this important restriction, if one of $\beta_\mu$ or $\beta_\mu^{LS}$ is optimal, so is the other given by Proposition \ref{propcgg1}.

The previous Proposition indicates that we cannot expect to obtain really new results, using weak log-Sobolev instead of weak Poincar\'e. For instance, if the Lyapunov condition \eqref{eqweaklyap} in Theorem \ref{lyapweakP} is satisfied, $\mu$ will satisfy some weak log-Sobolev inequality and get a rate function. If we are able to give a direct proof of this fact, all the direct proofs we know do not furnish better rates than the one obtained by transferring the one for the weak Poincar\'e inequality.

The second family of weakened logarithmic Sobolev inequalities is the
family of \emph{weighted logarithmic Sobolev} inequalities.
We say that $\mu$ satisfies a weighted logarithmic Sobolev inequality with weight $\omega$, $\omega$ being a non-negative function defined on $\mathbb{R}^d$ if there exists a constant $C_{LS,\omega}(\mu)$ such that for all nice function $f$
\begin{equation}\label{eqweightLS}
    \Ent_\mu(f^2)\le C_{LS, \omega}(\mu)\int \vert \nabla f\vert^2\omega^2 d\mu.
\end{equation}

It seems that these inequalities appeared in \cite{BLweight} in the study of generalized Cauchy measures, and then in \cite{BLweight2} for more general measures. A systematic study is done in \cite{CGWu}. We shall recall some of the results therein below.

The weighted logarithmic Sobolev inequality satisfies similar properties to the weighted Poincar\'e one.
\begin{proposition}\label{propweightedLSIproperties}
    It holds
    \begin{enumerate}
        \item[(1)] for all $x\in\mathbb{R}^d, C_{LS,\omega}(x+Z) = C_{LS,\tilde\omega}(Z)$, where $\tilde \omega(z) = \omega(z-x)$.
        \item[(2)] for any $\lambda\in\mathbb{R}, C_{LS,\omega}(\lambda Z) = \lambda^2C_{LS,\tilde\omega}(Z)$, where $\tilde \omega(z) = \omega(z/\lambda)$.
        \item [(3)] for any map $T:\mathbb{R}^d\to\mathbb{R}^d$ which is $L$-Lipschitz and invertible, $C_{LS,\omega}(T(Z)) = L^2C_{LS,\tilde\omega}(Z)$, where $\tilde \omega(z) = \omega(T^{-1}(z))$.
        \item[(4)] if $Z_1, \dots, Z_n$ are independent and $P_i$ is the projection on the $i$-th coordinate, then $\Ent\left(f^2\right)\leq \max_iC_{LS,\omega_i}(Z_i)\mathbb{E}\left[\sum_i\left\vert \nabla_{Z_i} f\left( Z_1,\dots, Z_n\right)\right\vert^2\omega_i^2\circ P_i(Z_1, \dots, Z_n)\right]$.
        \item[(5)] if $Z_1, \dots, Z_n$ are independent, $\Ent\left(f^2\left(\sum Z_i\right)\right)\leq\mathbb{E}\left[\left\vert \nabla f\left(\sum Z_i\right)\right\vert^2\sum_{i} C_{LS, \omega_i}(Z_i)\omega_i^2(Z_i)\right]$.
    \end{enumerate}
\end{proposition}

\begin{proof}
    The proof follows from the well known subadditivity property of the entropy (see e.g.Lemma 2.3.6 in \cite{Chewi}) and  Proposition \ref{propweightedPIproperties}.
\end{proof}

As for the Poincar\'e inequalities, a weighted logarithmic Sobolev inequality implies a weak logarithmic Sobolev inequality (Theorem 3.3 in \cite{CGWu}) as well as several other inequalities. \cite{CGWu} also introduces a Lyapunov type condition (see Theorem 2.1 therein) yielding explicit weights for the example of generalized Cauchy distributions.

\begin{example}\label{examplscauchy}
Let $\mu_\alpha(dx)= z_\alpha^{-1} \, (1+|x|^2)^{-(d+\alpha)/2}$ with $\alpha >0$. 

It is shown in \cite{CGWu} Corollary 2.9 that $\mu_\alpha$ satisfies a weighted logarithmic Sobolev inequality with optimal weight $$\omega_\alpha^2(x) = (1+|x|^2) \, \ln(e+|x|^2) \, ,$$ and some non explicit constant $C_{LS,\omega}(\mu_\alpha)$. Non explicit means that its value does not appear in \cite{CGWu} but can (with some difficulty) be traced in the proof of the Corollary, applied to the explicit case of $\mu_\alpha$. Actually the result is more general and applies to $\kappa$-concave distributions (see the formulation and the proof in subsection 2.3.1 of \cite{CGWu}). This proof is based on the ad-hoc Lyapunov criterion.

In the specific situation of $\mu_\alpha$, the very recent \cite{HugLS} uses a different approach based on $\Gamma_2$ calculus and recovers a slightly modified (but asymptotically equivalent)  weight and explicit constants, but apparently for a restricted range of $\alpha$'s (see Theorem 4.3 and Corollary 4.4 therein).

All these results improve on the original worse weight $\omega^2(x)=(1+|x|^2)^2$ in \cite{BLweight}.
\hfill $\diamondsuit$
\end{example}

We shall not discuss the case of Subbotin distributions for $\alpha <1$ only give some hints. First choosing $F(x)=e^{\gamma |x|^\alpha}$ for $\gamma$ small enough one can check that $F$ is a Lyapunov function in the sense of (2.1) in \cite{CGWu} (see Lemma 3.8 in \cite{CGGR}), with $\phi(u)=u\ln^{2(\alpha-1)/\alpha}(c+u)$ for some $c$. The weight can then be deduced from Proposition 2.6 and Theorem 2.1 in \cite{CGWu} following the lines of the proof for Cauchy distributions.

Similarly to Proposition \ref{propWeightPperturb} we now analyze how the weighted logarithmic Sobolev inequality behaves under perturbation.

\begin{proposition}\label{propWeightLSIperturb}
Let $\mu(dx)= e^{-V(x)}dx=e^{-(U+W)(x)} dx$ for smooth $U$ and $W$ be a probability measure. Denote $\nu(dx)=e^{-W(x)}dx$ supposed to be a probability measure too that satisfies a weighted logarithmic Sobolev inequality with weight $\omega$. 
\begin{enumerate}
     \item (Bounded perturbation) Assume that $U(x) \geq m_U$ (in other words $e^{-U}$ is bounded, notice that $m_U\leq 0$).

Then $$C_{LS, \tilde\omega}(\mu)\leq e^{-m_U} C_{LS, \omega}(\nu)\quad\textrm{with} \quad\tilde\omega=\omega e^{U/2} \, .$$
Furthermore, if $U$ is bounded, then $\mu$ satisfies a weighted logarithmic Sobolev inequality with constant $$C_{LS, \omega}(\mu)\leq e^{\Osc(U)} C_{LS, \omega}(\nu)\, .$$
\item (Unbounded perturbation)
\begin{enumerate}
    \item (Weighted Lipschitz case) Suppose that there exist $\varepsilon, \varepsilon'>0$ such that
    \begin{align*}
    s&=\sup_x C_{P, \omega}(\nu)\frac{1+\varepsilon}{4}|\nabla U|^2\omega^2<1\\
        \beta&=\sup_x C_{LS, \omega}(\nu)\frac{1+\varepsilon'}{4}|\nabla U|^2\omega^2<\infty
    \end{align*}
    then $\mu$ satisfies a weighted logarithmic Sobolev inequality with weight $\omega$ and constant
    $$C_{LS,\omega}(\mu)\le \frac{\alpha}{\alpha-1} \left[(1+\varepsilon^{-1}) C_{LS, \omega}(\nu)+\left(2+\beta + \frac{\int e^{\alpha U} d\mu}{\alpha}\right)\frac{(1+\varepsilon'^{-1})C_{P,\omega}(\nu)}{1-s}\right]$$
    for any $\alpha>1$.
    Recall that we have seen that a weighted logarithmic Sobolev inequality implies a weighted Poincar\'e inequality with the same weight.

    Alternatively, assuming that there exists $\varepsilon$ such that $s$ defined above satisfies $s<1$, then $\mu$ satisfies a weighted logarithmic Sobolev inequality with weight $\tilde\omega = \omega e^{U^+/2}/(\int e^{U^-}d\nu)^{1/2}$ and constant
    $$
    C_{LS, \tilde\omega}(\mu) \leq \left(\int e^{U^-}d\nu\right)\frac{1+\varepsilon^{-1}}{1-s}((2-s)C_{LS, \omega}(\nu) + (2 + M_U)C_{P, \omega}(\nu)), 
    $$
    where $U^+$ and $U^-$ denote the positive and negative part of $U$, respectively.
    \item (Weighted generator case) Let us introduce
    $$L_{U+W}^{\omega}f = \omega^2 \Delta f +(\nabla \omega^2-\omega^2 \nabla (U+ W)) .\nabla f .$$
    Suppose that
    \begin{align*}
    \beta &=\sup_x \frac{1}{2}C_{LS,\omega}(\nu)\left(\frac12|\nabla U|\omega^2+L^{\omega,\mu}U\right)_+<\infty\\
    s&=\sup_x \frac{1}{2}C_{P,\omega}(\nu)\left(\frac12|\nabla U|\omega^2+L^{\omega,\mu}U\right)_+<1  
    \end{align*}
    then $\mu$ satisfies a weighted logarithmic Sobolev inequality with weight $\omega$ and constant
    $$C_{LS,\omega}(\mu)\le\frac{\alpha}{\alpha-1}\left[ C_{LS,\omega}(\nu) + \left(2+\beta + \frac{\int e^{\alpha U} d\mu}{\alpha}\right)\frac{C_{P,\omega}(\nu)}{1-s}\right]$$
    \item (Weighted Lyapunov condition) Let us suppose that there exists $F\ge 1$, some point $x_0$ and constants $b,\theta>0$ such that
     \begin{equation}\label{lyapw_LSI}
    L_W^\omega F(x)\le -\theta F(x) +b 1_{|x|\le R}
    \end{equation}
    and there exists $\theta'<\theta$ such that
    $$-\omega^2\nabla U.\nabla F\le \theta' F.$$
    Besides assume that $\mu$ satisfies some local super Poincar\'e inequality, 
    then $\mu$ satisfies a weighted logarithmic Sobolev inequality with weight $\omega$.
\end{enumerate}
\end{enumerate}
\end{proposition}
The proof is provided in Appendix \ref{appendix_section_5}.

\newpage

\appendix

\section{Proofs of Section 2}\label{appendix_section_2}

\textbf{Proof of Proposition \ref{proppropweakP}}
\begin{proof}
(1), (2) and (5) are immediate, (3) is easy and shown in Theorem 5 of \cite{BCR2}. For (4) simply write
\begin{eqnarray*}
\Var (f(Z_1+Z_2)) &=& \int f^2(z_1+z_2) \, \mu_{Z_1}(dz_1) \, \mu_{Z_2}(dz_2) - \mathbb E^2(f(Z_1+Z_2)) \\ &=&  \int \left(\Var (f(Z_1+z_2)) + \mathbb E^2(f(Z_1+z_2))\right) \mu_{Z_2}(dz_2) - \mathbb E^2(f(Z_1+Z_2))\\&\leq& \int \left(\beta_{Z_1}(s) \, \int |\nabla f(z_1+z_2)|^2 \mu_{Z_1}(dz_1) + s \, \Osc^2(f)\right) \mu_{Z_2}(dz_2) \\ && \quad + \int \mathbb E^2(f(Z_1+z_2)) \, \mu_{Z_2}(dz_2) \, - \, \mathbb E^2(f(Z_1+Z_2))\\ &\leq& \beta_{Z_1}(s) \, \mathbb E[|\nabla f|^2(Z_1+Z_2)] + s \Osc^2(f)  \\ && \quad + \beta_{Z_2}(s) \, \int \, |\nabla \, \mathbb E(f(Z_1+.))|^2 \mu_{Z_2}(dz_2) + s \, \Osc^2(\mathbb E(f(Z_1+z_2)))
\end{eqnarray*}
and the result follows, first assuming that everything is regular enough for exchanging $\nabla$ and $\mathbb E$, then using Cauchy-Schwarz inequality and finally using that $\Osc(\mathbb E(f(Z_1+.))) \leq \Osc f$. One can then use an approximation procedure and (5) in order to conclude.
\end{proof}

\noindent \textbf{Proof of Lemma \ref{lemWPp}}
\begin{proof}
Let $f$ such that $\mu(f)=0$. For $M>0$, it holds 
\begin{eqnarray*}
\mu(f^2) &=& \int f^2 \, \mathbf 1_{|f|\leq M} \, d\mu \, + \, \int f^2 \, \mathbf 1_{|f|> M} \, d\mu \\ &\leq& \int |f\wedge M\vee -M|^2  \, d\mu \, + \, \int f^2 \, \mathbf 1_{|f|> M} \, d\mu \\ &\leq& \beta_\mu(s) \, \int |\nabla f|^2 d\mu + 4sM^2 + \left(\int \, f\wedge M\vee -M \, d\mu\right)^2 + \, \int f^2 \, \mathbf 1_{|f|> M} \, d\mu \\ &\leq& \beta_\mu(s) \, \int |\nabla f|^2 d\mu + 4sM^2 + \left(\int \, |f| \, \mathbf 1_{|f|> M} \, d\mu\right)^2 + \, \int f^2 \, \mathbf 1_{|f|> M} \, d\mu \\&\leq& \beta_\mu(s) \, \int |\nabla f|^2 d\mu + 4sM^2 + 2 \; \frac{||f||^{p}_{\mathbb L^p(\mu)}}{M^{p-2}}
\end{eqnarray*}
where we used Cauchy-Schwarz and H\"{o}lder inequalities. Choosing $M$ such that the two last terms are equal (which is optimal up to a factor $2$), i.e. $M=||f||_{\mathbb L^p(\mu)}/(2s)^{1/p}$ one thus obtain $$\Var_\mu(f) \leq \beta_\mu(s) \, \int |\nabla f|^2 d\mu + 2^{{3}-(2/p)} s^{1-(2/p)} \, ||f - \mu(f)||^2_{\mathbb L^p(\mu)} \, ,$$ hence the result.
\end{proof}

\newpage
\section{Proofs of Section 3}\label{appendix_section_3}
\noindent \textbf{Proof of Proposition \ref{propweightedPIproperties}}
\begin{proof}
    (1), (2) and (3) are immediate. For (4) and (5), one uses the following well known decomposition  result: if $Z_1, \dots, Z_n$ are independent random variables
    \begin{equation}\label{eq:subadditivity_variance}
        \Var(f(Z_1, \dots, Z_n))\leq \mathbb{E}\sum_k\Var(f(Z_1, \dots, Z_n)\vert Z_{-k}).
    \end{equation}
Using this, for (4) it follows
\begin{align*}
    \Var\left(f\left(Z_1, \dots, Z_n\right)\right)&\leq \mathbb{E}\sum_i\Var\left(f\left(Z_1, \dots, Z_n\right)\vert Z_{-i}\right)\\
    &\leq \mathbb{E}\sum_i C_{P, \omega_i}(Z_i)\mathbb{E}_{Z_i}\left[\left\vert \nabla_{Z_i} f\left( Z_1, \dots, Z_n\right)\right\vert^2\omega_i^2(Z_i)\Big\vert Z_{-i}\right]\\
    &\leq \max_iC_{P,\omega_i}(Z_i)\mathbb{E}_{( Z_1, \dots, Z_n)}\left[\sum_i\left\vert \nabla_{Z_i} f\left( Z_1,\dots, Z_n\right)\right\vert^2\omega_i^2\circ P_i(Z_1, \dots, Z_n)\right],
\end{align*}
where $P_i$ denotes the projection on the $i$-th coordinate.
On the other hand, for the convolution it follows that
\begin{align*}
    \Var\left(f\left(\sum Z_i\right)\right)&\leq \mathbb{E}\sum_i\Var\left(f\left(\sum Z_i\right)\vert Z_{-i}\right)\\
    &\leq \mathbb{E}\sum_i C_{P, \omega_i}(Z_i)\mathbb{E}_{Z_i}\left[\left\vert \nabla f\left(\sum Z_i\right)\right\vert^2\omega_i^2(Z_i)\vert Z_{-i}\right] \\
    &=\mathbb{E}\left[\left\vert \nabla f\left(\sum Z_i\right)\right\vert^2\sum_{i} C_{P, \omega_i}(Z_i)\omega_i^2(Z_i)\right] .
\end{align*}
\end{proof}

\noindent \textbf{Proof of Proposition \ref{propWeightPperturb}}
\begin{proof}
(1) In the first case, we have
    \begin{align*}
        \Var_\mu(f)&\leq \int(f-\mathbb{E}_\nu[f])^2 d\mu\leq e^{-m_U}\Var_\nu(f)\leq e^{-m_U}C_{P, \omega}(\nu)\int\vert\nabla f\vert^2\omega^2 d\nu \\&=e^{-m_U}C_{P, \omega}(\nu)\int\vert\nabla f\vert^2(\omega e^{U/2})^2 d\mu\, .
    \end{align*}
    On the other hand, when $U$ is bounded
        \begin{align*}
        \Var_\mu(f)&\leq e^{-m_U}C_{P, \omega}(\nu)\int\vert\nabla f\vert^2\omega^2 d\nu \leq e^{\Osc(U)}C_{P, \omega}(\nu)\int\vert\nabla f\vert^2\omega ^2 d\mu.
    \end{align*}
    
(2) (a) For all $a$ we have that $\Var_\mu(f)\le \int (f-a)^2d\mu$. We thus consider $a$ such that $\int (f-a)e^{-U/2}d\nu=0$, so that, by applying the weighted Poincar\'e inequality for $\nu$ with weight $\omega$, we get
\begin{eqnarray*}
    \int(fe^{-U/2}-ae^{-U/2})^2d\nu &\le&C_{P,\omega}(\nu) \int\left\vert\nabla \left((f-a)e^{-U/2}\right)\right\vert^2\omega^2d\nu\\
    &\le& C_{P,\omega}(\nu)(1+\varepsilon^{-1})\int|\nabla f|^2\omega^2d\mu\\&&\qquad + C_{P,\omega}(\nu)\frac{1+\varepsilon}{4}\int(f-a)^2|\nabla U|^2\omega^2d\mu\\
    &\le&C_{P,\omega}(\nu)(1+\varepsilon^{-1})\int|\nabla f|^2\omega^2d\mu+s\int(f-a)^2d\mu
\end{eqnarray*}
which leads to the conclusion if $s<1$.\\

(b) The starting point is the same except that we will exactly develop the square and use the weighted operator $L^{\omega, \mu}$ to deal with the middle term. First recall that for all nice function $f,g$
$$\int \, \langle\nabla f,\nabla g\rangle \, \omega^2 \, d\mu=-\int fL_{U+W}^{\omega}g\, d\mu.$$
We may then deduce
\begin{eqnarray*}
    \int(fe^{-U/2}-ae^{-U/2})^2d\nu &\le&C_{P,\omega}(\nu) \int|\nabla f- \frac12(f-a)\nabla U|^2\omega^2d\mu\\
    &=& C_{P,\omega}(\nu)\int|\nabla f|^2\omega^2d\mu+ \frac14C_{P,\omega}(\nu)\int(f-a)^2|\nabla U|^2\omega^2d\mu\\
    &&\qquad-\frac12C_{P,\omega}(\nu)\int \nabla (f-a)^2.\nabla U \omega^2d\mu\\
    &=&C_{P,\omega}(\nu)\int|\nabla f|^2\omega^2d\mu+ \frac14C_{P,\omega}(\nu)\int(f-a)^2|\nabla U|^2\omega^2d\mu\\
    &&\qquad+\frac12C_{P,\omega}(\nu)\int (f-a)^2 L_{U+W}^{\omega}U d\mu\\
    &\le& C_{P,\omega}(\nu)\int|\nabla f|^2\omega^2d\mu+s\int(f-a)^2d\mu
\end{eqnarray*}
from which we conclude if $s<1$.\\

(c) It is the same proof than the one for the usual Poincar\'e inequality recalled in the Section 2.
\end{proof}

\noindent \textbf{Perturbations of Subbotin distributions}
\medskip

We provide similar results to those in example \ref{exwPcauchyperturbed} for  Subbotin distributions.

Recall that the associated weight function is of the form $\omega^2(x) = 1+(1+|x|)^{2(1-\alpha)}$.
\begin{enumerate}
    \item[(a)] (Weighted Lipschitz case) The ``maximal'' perturbation satisfying the condition is of the form
    \begin{equation}\label{eqweightedPperturbsubbotin}
        U(x)-U(0) = \beta (1+\vert x\vert)^\alpha.
    \end{equation}
    \item[(b)] (Weighted generator case) Note that  potentials of the form \eqref{eqweightedPperturbsubbotin} also satisfy condition \eqref{eqconditionweightedlyapcaseperturb}.
    \item[(c)] (Weighted Lyapunov condition) Let $\nu(dx) = z_\alpha e^{-\vert x\vert^\alpha}$ for some $\alpha\in(0,1)$. The weighted generator is of the form
    \begin{equation*}
        L_W^\omega F := \omega^2 \Delta F +\left(2(1-\alpha) \frac{(1+\vert x\vert)^{1-2\alpha}}{\vert x\vert}       -\alpha\vert x\vert^{\alpha-2}\omega^2\right)\langle x, \nabla F \rangle\leq -\theta F.
    \end{equation*}
A natural choice is $F=e^{\gamma\vert x\vert^\alpha}$. We do not give all details, only informal computations:
    \begin{align*}
        \nabla F &= \gamma\alpha\vert x\vert^{\alpha-2} x F, \quad\quad \langle x, \nabla F \rangle=\gamma\alpha \vert x\vert^\alpha F, \\
        \Delta F &= \gamma^2\alpha^2\vert x\vert^{2(\alpha-1)}  F+\gamma\alpha(d+\alpha-2)\vert x\vert^{\alpha-2}F.
    \end{align*}
    Asymptotically we have
    \begin{align*}
        &\omega^2 \Delta F \sim \vert x\vert^{2(1-\alpha)}\Delta F \sim \gamma^2\alpha^2 F,\\
       &\left(2(1-\alpha) \frac{(1+\vert x\vert)^{1-2\alpha}}{\vert x\vert}       -\alpha\vert x\vert^{\alpha-2}\omega^2\right)\langle x, \nabla F\rangle\\
       &\quad\quad\sim\left(2(1-\alpha)\vert x\vert^{-2\alpha} -\alpha \vert x\vert^{\alpha-2}\vert x\vert^{2(1-\alpha)}\right) \gamma \alpha \vert x\vert^\alpha F\sim -\gamma\alpha^2 F.
    \end{align*}
Taking $\gamma\in(0,1)$, we have that the negative term dominates.

For a perturbation to be valid, there needs to exists $\theta'<\theta$ such that $-\omega^2\nabla U.\nabla F\le \theta' F$, this translates to 
    $$-\omega^2a\alpha\vert x\vert^{\alpha-2}\langle x, \nabla U\rangle\leq \theta'.$$
    This condition is satisfied for potentials of the form \eqref{eqweightedPperturbsubbotin}. Bounding $\omega^2$, we get the following simpler condition  for large values of $\vert x\vert$
    \begin{equation*}
        -\langle x, \nabla U\rangle\leq \frac{\theta'\vert x\vert^\alpha}{a\alpha}.
    \end{equation*}
\end{enumerate}

\section{Proofs of Section 4}\label{appendix_section_4}

\noindent \textbf{Proof of Lemma \ref{lem:cap-to-weighted}}
\begin{proof}
Define 
$$
f_+ = \max(0, f-m)\,,\quad f_- = \min(0, f-m).
$$
We have that 
$$
\int (f-m)w^2d\mu = \int f_+^2 w^2d\mu + \int f_-^2 w^2d\mu,
$$
where $\mu(\{f_+>0\}) \leq 1/2$ and $\mu(\{f_-<0\}) \leq 1/2$. It suffices to estimate one of the integrals since the treatment of the other is identical. 

Denote $A_k = \{x:2^k<f_+\}$ for $k\in\mathbb{Z}$. It follows that
\begin{align*}
\int f_+^2w^2 d\mu &= \sum_{k\in\mathbb{Z}} \int_{\{2^k<f_+\le 2^{k+1}\}} f_+^2w^2d\mu \leq  \sum_{k\in\mathbb{Z}} 2^{2(k+1)}\int_{\{2^k<f_+\le 2^{k+1}\}} w^2d\mu \\
&= \sum_{k\in\mathbb{Z}} 2^{2(k+1)}\nu(A_k)\le 4C \sum_{k\in\mathbb{Z}} 2^{2k} \,Cap_\mu(A_k).
\end{align*}
For each $k$, define
$$
\varphi_k(x) = \min\left(1, \frac{(f_+-2^{k-1})_+}{2^{k-1}}\right).
$$
It holds that $\varphi_k$ is admissible for $Cap_\mu(A_k)$,
$1_{A_k} \le \varphi_k\leq 1
$
and 
$$
\vert \nabla \varphi_k\vert = \frac{\vert \nabla f\vert}{2^{k-1}} 1_{\{2^{k-1}<f<2^k\}}.
$$
Hence 
$$
2^{2k-2}Cap_\mu(A_{k}) \le \int_{\{2^{k-1}<f<2^k\}}\vert \nabla f\vert^2 d\mu.
$$
Summing over $k$ and noting that the sets $\{2^{k-1}<f<2^k\}$ are disjoint, it leads to
$$
\int f_+^2 w^2 d\mu \le 16 \, C\sum_{k\in\mathbb{Z}} \int_{\{2^{k-1}<f<2^k\}}\vert \nabla f\vert^2 d\mu = 16C\int_{\{f_+>0\}} \nabla f\vert^2 d\mu\leq 16C\int \nabla f\vert^2 d\mu,
$$
which concludes the proof.
\end{proof}

\section{Proofs of Section 5}\label{appendix_section_5}
\noindent \textbf{Proof of Proposition \ref{propertieswLSI}}

\begin{proof}
    (1), (2) and (5) are immediate. For (3) and (4), we use the well known subadditivity of the entropy, that is, if $Z_1, \dots, Z_n$ are independent, then
    \begin{equation*}
        \Ent(f(Z_1, \dots, Z_n))\leq \mathbb{E}\sum_i\Ent(f(Z_1, \dots, Z_n)\vert Z_{-i}).
    \end{equation*}
Based on this, for (3), we have that
\begin{align*}
    \Ent\left(f^2\left(Z_1, \dots, Z_n\right)\right)&\leq \mathbb{E}\sum_i\Ent\left(f^2\left(Z_1, \dots, Z_n\right)\vert Z_{-i}\right)\\
    &\leq \mathbb{E}\sum_i \left(\beta_{Z_i}(s/n)\mathbb{E}_{Z_i}\left[\left\vert \nabla_{Z_i} f\left( Z_1, \dots, Z_n\right)\right\vert^2\vert Z_{-i}\right] + s/n\text{Osc}^2(f)\right) \\
    &\leq \max_i\beta_{Z_i}(s/n)\mathbb{E}_{( Z_1, \dots, Z_n)}\left[\sum_i\left\vert \nabla_{Z_i} f\left( Z_1,\dots, Z_n\right)\right\vert^2\right] + s\text{Osc}^2(f) \\
    &= \max_i\beta_{Z_i}(s/n)\mathbb{E}_{( Z_1, \dots, Z_n)}\left[\left\vert \nabla f\left( Z_1,\dots, Z_n\right)\right\vert^2\right] + s\text{Osc}^2(f).
\end{align*}
On the other hand, for the convolution it follows that
\begin{align*}
    \Ent\left(f^2\left(\sum Z_i\right)\right)&\leq \mathbb{E}\sum_i\Ent\left(f^2\left(\sum Z_i\right)\vert Z_{-i}\right)\\
    &\leq \mathbb{E}\sum_i \left(\beta_{Z_i}(s/n)\mathbb{E}_{Z_i}\left[\left\vert \nabla f\left(\sum Z_i\right)\right\vert^2\vert Z_{-i}\right] + s/n\text{Osc}^2(f)\right) \\
    &= \left(\sum_i\beta_{Z_i}(s/n)\right)\mathbb{E}_{\sum_i Z_i}\left[\left\vert \nabla f\left(\sum Z_i\right)\right\vert^2\right] + s\text{Osc}^2(f).
\end{align*}
\end{proof}

\newpage
\noindent \textbf{Proof of Proposition \ref{propWeightLSIperturb}}

\begin{proof}
(1) In the first case, using Lemma 5.1.7 in \cite{BakryGentilLedoux2014}, we have
    \begin{align*}
        \Ent_\mu(f)&\leq \int \left( f^2\log \frac{f^2}{\mathbb{E}_\nu[f^2]}-f^2+\mathbb{E}_\nu[f^2]\right) d\mu\leq e^{-m_U}\Ent_\nu(f)\\&\leq e^{-m_U} C_{LS, \omega}(\nu)\int\vert\nabla f\vert^2\omega^2 d\nu =e^{-m_U}C_{LS, \omega}(\nu)\int\vert\nabla f\vert^2(\omega e^{U/2})^2 d\mu\, .
    \end{align*}
    On the other hand, when $U$ is bounded
        \begin{align*}
        \Ent_\mu(f)&\leq e^{-m_U} C_{LS, \omega}(\nu) \int\vert\nabla f\vert^2\omega^2 d\nu \leq e^{\Osc(U)}C_{LS, \omega}(\nu)\int\vert\nabla f\vert^2\omega ^2 d\mu.
    \end{align*}

(2) (a) First assume that $\mathbb{E}_\mu[f]=0$. Recall the inequality for $\alpha>1$
$$\int f^2 Ud\mu\le \frac1\alpha\Ent_\mu(f^2)+\frac{\mu(f^2)}{\alpha}\int e^{\alpha U}d\mu.$$
By applying the weighted log-Sobolev inequality for $\nu$ with weight $\omega$, we get
\begin{align*}
\Ent_\mu(f^2) &= \int \left(f^2e^{-U}\log \frac{f^2e^{-U}}{\mathbb{E}_\nu[f^2e^{-U}]}\right)d\nu\, + \int f^2 Ue^{-U}  d\nu\\&
\leq  \Ent_\nu\left((fe^{-U/2})^2\right) + \frac{1}{\alpha}\Ent_\mu(f^2) + \frac{\Var(f)}{\alpha}\int e^{\alpha U} d\mu\\
&\leq C_{LS,\omega}(\nu) \int\left\vert\nabla \left(fe^{-U/2}\right)\right\vert^2\omega^2d\nu + \frac{1}{\alpha}\Ent_\mu(f^2) + \frac{\Var(f)}{\alpha}\int e^{\alpha U} d\mu\\
    &\le C_{LS,\omega}(\nu)(1+\varepsilon^{-1})\int|\nabla f|^2\omega^2d\mu + C_{LS,\omega}(\nu)\frac{1+\varepsilon}{4}\int f^2|\nabla U|^2\omega^2d\mu\\
    &\phantom{\le} + \frac{1}{\alpha}\Ent_\mu(f^2) + \frac{\Var(f)}{\alpha}\int e^{\alpha U} d\mu\\
    &\le C_{P,\omega}(\nu)(1+\varepsilon^{-1})\int|\nabla f|^2\omega^2d\mu+\Var_{\mu}(f)\left(\beta + \frac{\int e^{\alpha U}d\mu}{\alpha}\right)+\frac{1}{\alpha}\Ent_\mu(f^2)
\end{align*}
where we can apply the result in (2a) Proposition \ref{propWeightPperturb} to $\Var_{\mu}(f)$. Finally, if $\mathbb{E}_\mu[f]\neq0$, the result follows by applying Rothaus lemma
$$
\Ent_\mu(f^2)\leq \Ent_\mu((f-\mathbb{E}_\mu[f])^2)+ 2\Var_\mu(f).
$$
For the second part, first assume that $U\leq M_U$ is upper bounded. We thus have
\begin{align*}
\Ent_\mu(f^2) &= \int \left(f^2e^{-U}\log \frac{f^2e^{-U}}{\mathbb{E}_\nu[f^2e^{-U}]}\right)d\nu\, + \int f^2 Ue^{-U}  d\nu\\&
\leq  \Ent_\nu\left((fe^{-U/2})^2\right) + M_U \int f^2 d\mu\\
    &\le C_{LS,\omega}(\nu)(1+\varepsilon^{-1})\int|\nabla f|^2\omega^2d\mu + \left(M_U + C_{LS,\omega}(\nu)\frac{1+\varepsilon}{4}\sup_x(|\nabla U|^2\omega^2)\right)\int f^2d\mu
\end{align*}
Applying Rothaus lemma and Proposition \ref{propWeightPperturb} (2) yields the following bound
\begin{align*}
    C_{LS, \omega}(\mu)&\leq C_{LS, \omega}(\nu)(1+\varepsilon^{-1}) + ((2+M_U)C_{P, \omega}(\nu) + C_{LS, \omega}(\nu))\frac{1+\varepsilon^{-1}}{1-s}\\
    & =\frac{1+\varepsilon^{-1}}{1-s} ((2-s)C_{LS, \omega}(\nu) + (2+M_U)C_{P, \omega}(\nu)).
\end{align*}
In general, decompose $U = U^+-U^-$ positive and negative part. Define  $M_U = \ln\left(\int e^{U^{-}}d\nu\right)$ and consider the probability measure
$$
d\eta =e^{-M_U+U^-}d\nu =e^{-\bar{U}}d\nu
$$
satisfying $\bar{U}\leq M_U$. Since we have that $|\nabla U^-|\leq |\nabla U|$, we can apply what precedes to get $C_{LS, \omega}(\eta)$ and then use (1) since $d\mu =e^{M_U-U^+}d\eta = e^{-\tilde U}d\eta$ with $\tilde U\geq -M_U$. This leads to
$$
    C_{LS, \tilde\omega}(\mu) \leq e^{M_U}\frac{1+\varepsilon^{-1}}{1-s}((2-s)C_{LS, \omega}(\nu) + (2 + M_U)C_{P, \omega}(\nu)), 
    $$
where $\tilde\omega = \omega e^{(U^+-M_U)/2} = \omega e^{U^+/2}/(\int e^{U^-}d\nu)^{1/2}$.\\

(b) As before, first assume $\mathbb{E}_\mu[f]=0$. In this case, we develop the square and use the weighted operator $L^{\omega, \mu}$ to deal with the middle term.
We may then deduce

\begin{eqnarray*}
    \Ent_\mu(f^2)  &\le&C_{LS,\omega}(\nu) \int|\nabla f- \frac{f}{2}\nabla U|^2\omega^2d\mu+ \int f^2 Ue^{-U}  d\nu\\
    &=& C_{LS,\omega}(\nu)\int|\nabla f|^2\omega^2d\mu+ \frac14C_{LS,\omega}(\nu)\int f^2|\nabla U|^2\omega^2d\mu\\
    &&-\frac12C_{LS,\omega}(\nu)\int \nabla f^2.\nabla U \omega^2d\mu + \frac{1}{\alpha}\Ent_\mu(f^2) + \frac{\Var(f)}{\alpha}\int e^{\alpha U} d\mu\\
    &=&C_{LS,\omega}(\nu)\int|\nabla f|^2\omega^2d\mu+ \frac12C_{LS,\omega}(\nu)\int f^2\left(\frac12|\nabla U|^2\omega^2 + L^{\omega,\mu}U\right)d\mu\\
    && + \frac{1}{\alpha}\Ent_\mu(f^2) + \frac{\Var(f)}{\alpha}\int e^{\alpha U} d\mu\\
    &\le& C_{LS,\omega}(\nu)\int|\nabla f|^2\omega^2d\mu+ \Var_\mu(f)\left(\beta + \frac{\int e^{\alpha U} d\mu}{\alpha}\right) + \frac{1}{\alpha}\Ent_\mu(f^2) 
\end{eqnarray*}
from which we conclude by applying the result in (2b) Proposition \ref{propWeightPperturb} and Rothaus lemma.
\\

(c) By the Lyapunov condition, $\mu$ satisfies a super weighted Poincar\'e inequality, as well as a weighted Poincar\'e inequality with weight $\omega$. 
To conclude, we note that a super weighted Poincar\'e inequality is equivalent to a defective weighted log-Sobolev inequality \cite{CGWu}. Combined with the weighted Poincar\'e inequality, this defective inequality can be transformed into a tight weighted log-Sobolev inequality.
\end{proof}

\section*{Acknowledgments}
A.~Guillin is supported by the ANR-23-CE-40003, Conviviality and has benefited from a government grant managed by the Agence Nationale de la Recherche under the France 2030
investment plan ANR-23-EXMA-0001.
P.~Cordero-Encinar is supported by EPSRC through the Modern Statistics and Statistical Machine Learning (StatML) CDT programme, grant no. EP/S023151/1.

\bigskip
\bigskip

\bibliographystyle{plain}
\bibliography{CCGweak}

\end{document}